\documentclass[12pt]{article}
\usepackage{amsmath,amssymb}
\usepackage{color}
\usepackage{geometry}
 \usepackage{ulem}
\usepackage[pdfdisplaydoctitle,colorlinks,breaklinks,urlcolor=blue,linkcolor=blue,citecolor=blue]{hyperref}

\usepackage{pdfsync}
\usepackage{amsfonts,latexsym,amsmath,amsthm,amscd,,amssymb, dsfont}
\usepackage[utf8]{inputenc}
\usepackage{color}
\usepackage{todonotes}
\usepackage{pdfsync}
\usepackage[notref,notcite]{showkeys}

\definecolor{red}{rgb}{1.0,0.0,0.0}

\definecolor{blu}{rgb}{0.0,0.0,1.0}

\definecolor{gre}{rgb}{0.03,0.50,0.03}
\def\gre#1{{\textcolor{gre}{#1}}}
\definecolor{amethyst}{rgb}{0.6, 0.4, 0.8}

\definecolor{blue-violet}{rgb}{0.54, 0.17, 0.89}

\definecolor{darkviolet}{rgb}{0.58, 0.0, 0.83}

\def\sqr#1#2{{\vcenter{\vbox{\hrule height .#2pt \hbox{\vrule
 width .#2pt height#1pt \kern#1pt \vrule
width .#2pt} \hrule height .#2pt}}}}

\def\qedo{\hbox{\hskip 6pt\vrule width6pt height7pt
depth1pt  \hskip1pt}\bigskip}

\def\E{\mathbb E}
\def\P{\mathbb P}

\def\bR{\mathbb R}
\def\R{\mathbb R}

\def\calh{\mathcal H}
\def\cK{\mathcal K}
\def\call{\mathcal L}

\def\cald{\mathcal D}

\def\<{\left\langle }
\def\>{\right\rangle }
\newtheorem{theorem}{Theorem}[section]
\newtheorem{lemma}[theorem]{Lemma}
\newtheorem{proposition}[theorem]{Proposition}

\newtheorem{definition}[theorem]{Definition}
\newtheorem{example}[theorem]{Example}
\newtheorem{remark}[theorem]{Remark}
\newtheorem{assumption}[theorem]{Assumption}
\newtheorem{notation}[theorem]{Notation}
\newtheorem{problem}[theorem]{Problem}

\numberwithin{equation}{section}

\newcommand{\radon}{\ensuremath{\mathcal M}}

\newcommand{\nada}[1]{}

\begin{document}

\title{Robust portfolio choice with sticky wages}
\author{Sara Biagini\thanks{-* sbiagini@luiss.it, fgozzi@luiss.it --  Department of Economics and Finance, Luiss University, Italy.}, \hspace{0.1cm}  Fausto Gozzi* and Margherita Zanella \thanks{ margherita.zanella@polimi.it -- Politecnico di Milano, Italy.}} \maketitle

 \begin{abstract}

 We present a robust version of the life-cycle optimal portfolio choice problem in the presence of labor income, as introduced in Biffis, Gozzi and Prosdocimi \cite{BGP} and Dybvig and Liu \cite{DYBVIG_LIU_JET_2010}. In particular, in \cite{BGP}  the influence of past wages on the future ones is modelled linearly in the evolution equation of labor income, through a given weight function. The optimisation relies on the resolution of  an infinite dimensional HJB equation. \\
 \indent   We improve the state of art in  three ways. First, we allow the weight to be a  Radon measure. This accommodates for more realistic weighting of the sticky wages, like, e.g., on a discrete temporal grid according to some periodic income.  Second, there is a  general correlation structure between labor income and stocks market. This naturally affects the optimal hedging demand, which may increase or decrease according to the correlation sign.  Third, we allow the weight to change with time, possibly lacking  perfect identification. The uncertainty is specified by a  given set of Radon measures $K$, in which the weight process takes values. This renders the inevitable uncertainty on how the past affects the future, and includes the standard case of error bounds on a specific estimate for the weight.
Under uncertainty averse preferences, the decision maker takes a maxmin approach to the problem.  Our analysis confirms the intuition: in the infinite dimensional setting, the optimal policy remains the best investment strategy under the worst case weight.\\

\\
\noindent{\bf Acknowledgements:} We warmly thank  Enrico Biffis and Giovanni Zanco for fruitful conversations on the topic. We are grateful to the referees and to the editors for their suggestions and  careful  reading, which helped us to substantially improve the paper.
 \end{abstract}

\tableofcontents

\section{Introduction}
\label{sez:1}

The goal of the present paper is to solve the life-cycle optimal portfolio choice problem of an agent, who allocates her wealth to risky assets and a riskless bond. The feasible allocations satisfy a borrowing constraint against future labor income.  Labor income dynamics here incorporate the stickiness feature of the wages, also called nominal rigidity.
As systematically outlined by  Keynes   in the article \emph{The General Theory of Employment, Interest and Money} (1936),   wages and prices do not adjust immediately to shocks in the economy.  A vast literature on the topic has followed since then, and we refer the reader to \cite{BGP} for a comprehensive list  of relevant papers.  \\
\indent Coming to modeling aspects, there is empirical evidence that ARMA processes, with their memory, offer satisfactory models for stochastic labor income (see e.g. \cite{Abowd-Card}, \cite{Mao-Sabanis}, \cite{McLaughling}).  However, solutions to some classes of stochastic delay differential equations (SDDE)  can be seen as weak limits of ARMA processes as shown in \cite{DUNS-Go-Tran,LORENZ_2006,REISS_2002,TranPHD}.
Therefore, SDDEs allow for a more realistic description of the  labor income evolution in continuous time models. \\

\indent Tractability is an issue, especially when looking for explicit solutions. The simplest possible delay equations are linear ones. In the context of portfolio selection with sticky labor income, this is the choice made in \cite{BGP}.
The authors model the stocks  by a multidimensional geometric Brownian motion. The labor income is perfectly correlated to the stock market, and  follows a linear SDDE in which the delay is present   in the drift.  The delay  term is given by  past wages weighted
on a \emph{bounded} time window $[-d, 0]$. In fact, there is evidence of  bounded memory in labor income adjustment delays, see \cite{MEGHIR_PISTAFERRI_2004}. The weighting is made according to a measure $\phi$, which is absolutely continuous wrt the Lebesgue measure on $[-d,0]$.   As for investment possibilities, the  agent is allowed to borrow against future wages and therefore the budget constraint is on total wealth, namely current financial wealth plus the present value of future labor income.  Given the budget constraint,  the authors in the cited \cite{BGP} are able to find an explicit solution to expected power utility maximisation from consumption and bequest.  The methodology relies on the resolution of an infinite dimensional HJB equation. The solution structure shows that this problem  can be seen as an infinite dimensional version of the classic Merton problem, as it is in fact Markovian in the current wealth and  in the labor income present and past path. We will refer to the optimization problem in \cite{BGP} as to \emph{the infinite dimensional  Merton problem}.\\

\indent We expand the above framework   in three ways.  First, the weight $\phi$ given to the past is   here  a general Radon measure on $[-d,0]$, not necessarily absolutely continuous with respect to  the Lebesgue measure.  This accommodates for more realistic weighting, like e.g. on a discrete temporal grid according to a specific, periodic income.
  The optimal policy (consumption, bequest and hedging) we find  is a constant proportion policy with respect to total wealth, \emph{modulo a hedging demand correction} due to the presence of labor income. Second, we allow for a general correlation structure between the driving multidimensional noises in the stock market and labor income respectively. According to the dependence of the labor income on the stock market shocks,  we analytically recover   variations in the hedging demand, which is an empirically observable phenomenon. In fact, the higher the correlation is, the lower the hedging demand:  diversification becomes  an issue and so the investor is more prudent.  This  result is in line with what obtained in the pioneering work \cite{VICERIA} in a discrete time setup. On the other hand, the papers \cite{ BGP}  and \cite{DYBVIG_LIU_JET_2010} (the latter in  the no delay case) treat only the case with perfect correlation, thus obtaining a \emph{negative} hedging demand correction. \\
\indent  Third, we allow the  delay measure $\phi$ to depend on $(\omega, t)$ and to take values  in a given set $K$ of Radon measures.  Alternatively said, we take into account possible lack of information on $\phi$,  due to the complexity of the global  dynamics in the economy and the set $K$ represents the confidence set.   There is an extensive literature on portfolio selection under uncertainty, see e.g \cite{BP} and the references there cited, \cite{NN}, \cite{riedel}.  However, to our knowledge  the present paper is the first to incorporate uncertainty in portfolio selection with labor income.  Assuming an uncertainty averse investor (see \cite{macche} for a definition),   she will first minimize over the delay measures  and then optimize over strategies.\\
\indent  Under the assumption that the set $K$ has an order minimum $\nu$, we show that the investor becomes observationally equivalent to one facing the infinite dimensional Merton problem under the  minimum  delay measure $\nu$. The  optimal strategy  is  found as the  solution to an infinite dimensional HJB equation. This intuitive result relies on a   non trivial existence and uniqueness result for the labor income equation with measure-valued  stochastic delay, proved in Appendix B.\\

 The article is organized as follows.
Section 2  describes the model  when $\phi$ is a fixed (namely, not time dependent) Radon measure on the given time window.
Section 3 is dedicated to the solution of the agent's infinite dimensional optimization problem, which generalizes the one given in \cite{BGP}. There, we illustrate the main novelties, consisting in the study of the driving operator $A_\phi$ and its adjoint (see in particular Subsection 3.2).
In Section 4, the robust setting comes into play and we solve the minimax version of the  infinite dimensional problem. An Appendix  collects a few  technical results and  concludes the paper.

\section{The model}
\subsection{The state equations}
 The exposition concerns the case in which stocks and labor income are perfectly correlated, as in \cite{BGP}. Such choice allows for a more direct  comparison  with respect to the results in  \cite{BGP}.  As anticipated in the Introduction, we also handle the general correlation case, which is quite relevant from an economical viewpoint.  We  illustrate how to do it in Remark \ref{General-correlation} and  in Remark  \ref{VICEIRA}, where we provide the  general form of the  optimal hedging demand. \\
\indent Consider a filtered probability space $(\Omega, \mathcal F, \mathbb F,  P)$, and an $\mathbb F$-adapted vector valued process $(S_0,S)$ representing the price evolution of a riskless asset, $S_0$, and $n$ risky assets, $S=(S_1,\ldots,S_n)^\top$.  Their  dynamics are
\begin{eqnarray}\label{DYNAMIC_MARKET}
\left\{\begin{array}{ll}
dS_0(t)= S_0(t) r  dt\\
dS(t) =\text{diag}(S(t)) \left(\mu dt + \sigma dZ(t)\right)\\
S_0(0)=1\\
S(0)\in {\mathbb R}^n_{+},
\end{array}
\right.
\end{eqnarray}
in which $Z$ is a $n$-dimensional Brownian motion and $\mathbb F$ is its  augmented natural filtration. The interest rate $r>0$ is  constant, the drift $\mu \in \mathbb R^n$ and the volatility $\sigma$ is an invertible matrix in  $\mathbb R^{n \times n} $. \\
 A representative agent  is endowed with initial wealth $w\geq 0$. She is active on the market and receives wages till her death.  The time of death $\tau_{\delta}$ is modeled as an exponential random variable of parameter $\delta >0$.  Further  assumptions are listed here below.
\begin{assumption}\label{hp:tau}

\begin{itemize}
  \item[(i)]
The death time $\tau_{\delta}$ is independent of  $Z$.

  \item[(ii)]  The reference filtration is the minimal enlargement of the Brownian filtration satisfying the usual assumptions and making $\tau_\delta$ a stopping time.  From  \cite[Section IV.3]{Protter}, the filtration is $\mathbb G := ( \mathcal G_t \big)_{t \ge 0}$  with the sigma-field $\mathcal G_t$ given by
\begin{equation*}
 	\mathcal G_t:=   \mathcal F_t \vee  \sigma (\tau_{\delta}\wedge t),
\end{equation*}
in which $\sigma(X)$ denotes as usual the sigma-field generated by a random variable $X$.  We remark that the filtration $\mathbb{G}$ is  complete and right continuous.
\footnote{The filtration generated by two right continuous filtrations is not right continuous in general. Instead of $\mathbb{G}$, one should consider $\mathbb{G}^*$ defined as:
$ \mathcal G^*_t:= \bigcap_{u>t} \, \  \mathcal F_u \vee  \sigma (\tau_{\delta}\wedge u) $.  However here the processes $Z(\cdot)$ and $\tau_{\delta}\wedge \cdot$ are independent and the filtrations $\mathbb{F}$ and $\sigma (\tau_{\delta}\wedge t)$ are right continuous.  Hence   Proposition 1.12 in \cite{AKSAMITJEANBLANC17} applies, and
  $\mathbb{G} =  \mathbb{G}^*$.}
\end{itemize}
\end{assumption}

By \cite[Proposition 2.11-(b)]{AKSAMITJEANBLANC17}, if a process $A$ is ${\mathbb G}$-predictable then there exists a process $\widetilde{A}$
which is ${\mathbb F}$-predictable and
 \begin{equation}\label{eq:GFpred}
 \text{for } P-a.s. \  \omega,  \qquad   A(s,\omega) = \widetilde{ A}(s,\omega) \qquad
 \forall s \in [0,\tau_\delta(\omega)]
\end{equation}
 Namely,
 $P$-a.s. $\widetilde{ A}$ has the same path as $A$ till the death time $\tau_{\delta}$.  Hence $A$ and $\widetilde{A}$ are indistinguishable up to $\tau_\delta$.
Therefore, we may and do work with ${\mathbb F}$-predictable,  pre-death versions of the processes involved, state variables and controls.\\
\indent The wealth of the agent at time $t\ge 0$ is denoted by $W(t)$  and the wage rate is $y(t)$. She can invest  in the riskless and risky assets, and can consume at a  rate $ c(t)\geq 0$. The wealth allocated to the risky assets is $ \theta(t)\in \mathbb R^n$  at each time $t\geq 0$.  The agent has a bequest target $B(\tau_\delta)$ at  death, where the bequest  process $ B(\cdot)\geq 0$  is also chosen by the agent. To cover the gap between bequest and wealth at death: $$ B(\tau_{\delta})-W(\tau_{\delta}),$$ the agent pays an instantaneous life insurance premium of  $\delta( B(t)-W(t))$ for $t<\tau_\delta$. As in \cite{DYBVIG_LIU_JET_2010}, we interpret a negative  value $ B(t)- W(t)<0$ as a  life annuity trading wealth at death for receiving a positive income flow $\delta(W(t)- B(t))$ while living.

\begin{notation}
\label{not:basic} Let us introduce the following notations:
\begin{itemize}
\item
$\call^1 = L^1_{loc} (\Omega \times \mathbb R_+, \mathbb{F},P\otimes dt ;
\mathbb R)$,
$\call^1_+ = L^1_{loc} (\Omega \times \mathbb R_+, \mathbb{F},P\otimes dt ;
\mathbb R_+)$,
$\call_n^2 := L^2_{loc}(\Omega \times \mathbb R_+, \mathbb{F},P\otimes dt;\mathbb R^n)$. Here, the subscript ``loc'' refers to the time variable, e.g. $f \in  L^1_{loc} (\Omega \times \mathbb R_+, \mathbb{F},P\otimes dt ;\mathbb R)$ means that
$f\in L^1 (\Omega \times[0,T], \mathbb{F},P\otimes dt ;\mathbb R)$ for all $T>0$.
  \item
 $\Pi^0:=  \Big\{ (c, B,\theta) \mid c \text{ and } B \text{ belong to }  \call^1_+, \text{ and } \theta \in \call^2_n \Big\}$. The triplets in   $\Pi^0$ are denoted by  $\pi$.
\item $\mathbf 1 = (1,\dots, 1)^\top$, the unitary vector in $\mathbb R^n$.
\item
 The space of deterministic functions on $[-d,0)$ which are  square integrable with respect to the Lebesgue measure $ds$ is denoted by $L^2([-d,0);\mathbb{R})$,   $ds$ being understood.\footnote{ It is well known that $L^2([-d,0);\mathbb{R})$ is the same space as $L^2((-d,0);\mathbb{R})$ or $L^2([-d,0];\mathbb{R})$ which are usually written simply as
$L^2(-d,0;\mathbb{R})$. Here we prefer to keep the above notation as it is coeherent with the fact that $[-d,0)$
is the set of past times in our setting.}
\item
The space of Radon measures on $[-d,0]$ that are null in $0$ (i.e. non atomic at $0$) is denoted by $\mathcal M^0$.  We recall that a  Radon measure on the compact  {set} $[-d,0]$  is simply a Borel regular measure. The reader who is not familiar with these notions is referred to  \cite{AB}.
\end{itemize}
\end{notation}

 When $(c, B, \theta) \in \Pi^0$, the stochastic differential equation (SDE) for the wealth $W$ is the equation for the Merton  optimal investment-consumption problem. As
  in \cite{DYBVIG_LIU_JET_2010,BGP,DGZZ} the drift incorporates
the wage rate $y(t)$ and the life insurance $\delta(B(t)-W(t))$:
\begin{equation}\label{eq:stateW}
dW(t)=\left[W(t) r + \theta(t)^\top (\mu-r\mathbf{1})  + y(t) - c(t)
-\delta\left(B(t)-W(t)\right)\right] dt + \theta(t)^\top \sigma dZ(t)
\end{equation}
with the initial condition $W(0)=w\in \R$. \\
\indent The state equation for the labor income rate $y$  is more delicate, as it is a Stochastic Delayed Differential Equation, SDDE as a shorthand (also called Path-Dependent SDE). In the present context, an SDDE for $y$ was first introduced in \cite{BGP} as an innovation with respect to the standard linear SDEs approach used  e.g. in \cite{DYBVIG_LIU_JET_2010}, so to  render the empirical ``stickiness'' feature of the wages. The dynamics  of $y$ there were those of a geometric Brownian motion with driving noise $Z$, drift $\mu_y \in \mathbb{R}$ and volatility $\sigma_y \in  \mathbb R^n$ plus an additional, linear delay term in the drift.  In \cite{BGP} \gre{(and later also in \cite{DGZZ})}, such delay term was given by the integral of the past path of $y$ with respect to  a measure $\phi$  absolutely continuous wrt $ds$. Here, we allow $\phi$ to be a   measure in  $\mathcal M^0$:
\begin{align}\label{DYNAMICSy}
\begin{split}
\left\{\begin{array}{ll}
dy(t) = & \left[ y(t) \mu_y+\int_{-d}^0 y(t+s)\phi(ds)   \right] dt + y(t)\sigma_y^\top    dZ(t),\\[2mm]
y(0)= & x_0, \quad y(s) = x_1(s) \mbox{ for $s \in  [-d,0)$},
\end{array}\right. \end{split}
\end{align}
in which the initial data are $x_0 \in \R $ and $x_1\in L^2([-d,0);\R)$.
  To our knowledge, in the extant literature there is no result of well posedness for such SDDE.  The main issue here is to give sense, in the stochastic case\footnote{In the deterministic case this has been done e.g. in
\cite[Chapter 4]{BENSOUSSAN_DAPRATO_DELFOUR_MITTER}.}, to the integral $\int_{-d}^{0} y(t+s)\phi(ds)$
for $t<d$ since it involves the integral in $\phi$ of the function $x_1$ which is only defined $ds$-a.e.
This issue becomes more evident in the subsequent Section 4. There, to accommodate for robustness we allow $\phi$ to depend on $(\omega,t)$.  The corresponding state equation for labor income becomes \eqref{EV-y-phi()}. The Appendix \ref{app:wellposed} (in particular Proposition \ref{Prop-SDDE-app}) is precisely dedicated to the proof of  the  well posedness  of the general \eqref{EV-y-phi()}, and the continuity of the solution wrt the initial data.
The well posedness of \eqref{DYNAMICSy} above then follows as a particular case of such more general result.
For the reader's convenience we provide here the well posedness result needed now and in the subsequent
Sections 2.2, 2.3 and 3.

\begin{proposition}
\label{pr:statelabor}
Let $\phi\in \mathcal M^0$. Then, for every $x_0 \in \R $ and $x_1\in L^2([-d,0);\R)$
the SDDE \eqref{DYNAMICSy} admits
a unique strong(in the probabilistic sense) solution in $\mathcal{L}^2_1$ with $P$-a.s. continuous paths.
\end{proposition}


\begin{remark}
Note that the drift of the equation (and the data) is split into the contribution of the present $y(t)$ and of the past path of  $y$ on $[-d, 0)$ for ease of presentation, as explained in Remark \ref{rm:phinonatomic} below.
Note also that the integrability condition on $x_1$ is required with respect to the \emph{Lebesgue measure} $ds$, and not with respect to $\phi(ds)$, as it would seem natural at first sight.
This requirement stems from our main goal, which is  the robust case treated  in Section 4.  When  $\phi$ varies it seems more natural to consider integrability of the datum $\phi$ wrt to the standard reference measure $ds$.

\end{remark}

Once we know, from Proposition \ref{pr:statelabor} above
the properties of the solution to \eqref{DYNAMICSy},
the existence and uniqueness of the strong solution to the SDE for $W$  are immediate
(see e.g. \cite[Section 5.6]{KARATZAS_SHREVE_91}).
\hfill\qedo
\begin{remark}
\label{rm:phinonatomic}
 The state equation for $y$ can be equivalently reformulated without splitting in the drift    the contribution of past an present.  Note first that since $\phi$ is non atomic at $0$, in \eqref{DYNAMICSy} the  $\phi$-average of the past path on the window $[-d, 0)$  coincides with the $\phi$-average on  the whole path on $[-d,0]$. By the way, this is  why    we did not specify whether in $\int_{-d}^0 y(t+s) \phi(ds)$ the second extremum of integration was taken or not. And in the following we will sometimes refer to $[-d, 0)$ and sometimes to $[-d,0]$  to indicate the delay time window, without distinction. \\
\indent Now, the equivalent formulation without splitting  is performed by simply adding to $\phi$ an atom in $\{0\}$ of size $\mu_y$:  $$ y(t) \mu_y+\int_{-d}^0 y(t+s)\phi(ds) =  \int_{-d}^0 y(t+s) d\psi(s), \  \  \text{ where } \psi = \phi + \mu_y \delta_{0}  $$

\hfill\qedo
\end{remark}

 \begin{notation}
\label{not:solWy}
The unique solution of the labor income equation at time $t\ge 0$ will be denoted by $y^{x}(t;\phi)$ to underline its dependence on the initial datum $x:=(x_0,x_1)$ and on the measure $\phi$. When no confusion is possible, we w
ill omit $\phi$ and/or $x$ in the superscript.
Similarly, given $\pi\in \Pi^0$ we will denote the unique solution to the wealth equation by $W^{w,x}(t;\phi,\pi)$.  Sometimes, the superscripts $w,x$, and/or the arguments  $\pi$ and $\phi$  will be omitted as well.
\end{notation}

\subsection{The optimization problem}
The agent is risk averse, with a power HARA  utility function with  parameter $\gamma \in (0,1) \cup (1, \infty)$.  As in  \cite{BGP}, given the controls $\pi=\left(c,B,\theta\right)\in \Pi^0$, her expected utility from lifetime consumption and bequest is
\begin{equation}\label{exp-ut}
J(c,B):=\mathbb E \left(\int_{0}^{\tau_{\delta}} e^{-\rho t }
\frac{c(t)^{1-\gamma}}{1-\gamma} dt
+ e^{-\rho \tau_{\delta} } \frac{\big(k B(\tau_\delta)\big)^{1-\gamma}}{1-\gamma}
\right),
\end{equation}
where $k>0$ weights the money left at death and $\rho >0$ is the time impatience coefficient.
Since the death time is independent of $Z$ and exponentially distributed, by an application of Fubini Theorem the objective $J$ can be rewritten as
\begin{eqnarray}\label{gaininfty}
\mathbb E \left(\int_{0}^{\infty} e^{-(\rho+ \delta) t }
\left( \frac{c(t)^{1-\gamma}}{1-\gamma}
+ \delta \frac{\big(k B(t)\big)^{1-\gamma}}{1-\gamma}\right) dt
\right).
\end{eqnarray}
  The agent  will receive a  contingent claim $\Psi$ with maturity $T$ only if she survives that date.   So, if $Q$ denotes the unique martingale measure, then
  $$ \text{price} (\Psi) = E^Q[e^{-rT} \Psi\,  I_{\{T< \tau_{\delta}\}}]. $$
 Given that $\tau_\delta$ is independent of the market drivers, and $P(T< \tau_{\delta}) =e^{-\delta T}$,  the pre-death state-price density $\xi $  of the agent solves
\begin{equation}\label{PRICE_DENSITY}
\left\{\begin{array}{ll}
d \xi (t)& = - \xi(t)(r +\delta) dt  -\xi(t) \kappa^\top dZ(t),\\
\xi(0)&=1,
\end{array}\right.
\end{equation}
where $\kappa$ is the market price of risk
\begin{equation}\label{DEF_KAPPA}
\kappa:= (\sigma)^{-1} (\mu- r \mathbf 1).
\end{equation}
 Following \cite{BGP} and \cite{DYBVIG_LIU_JET_2010}-Problem 1, the agent is allowed to borrow against future income.
The budget constraint is then nonnegativity of the total wealth $\Gamma$, evaluated as current wealth $W$ plus present value of future labor income. More precisely, given $w\in \R$, $x\in\R\times L^2([-d,0);\R)$, $\phi \in \mathcal{M}^0$, $\pi\in \Pi^0$, we define
\begin{equation}\label{budget-constraint}
\Gamma^{w,x}(t;\phi,\pi) := W^{w,x}(t;\phi,\pi)
+ \xi^{-1}(t)\mathbb E\left( \int_t^{\infty} \xi(u) y^{x}(u;\phi) du \Bigg\vert \mathcal F_t\right)  \geq 0,
\end{equation}
omitting $w,x,\phi,\pi$ when it's clear from the context.\\
\indent The admissible controls are then the sufficiently integrable triplets $\pi$ such that the associated total wealth $\Gamma$ is nonnegative:
\begin{equation}\label{DEF_PI_FIRST_DEFINITION}
\Pi(w,x;\phi):= \Big\{ \pi\in \Pi^0
 \mid
\Gamma^{w,x}(t;\pi,\phi)\geq 0\,\quad \forall t \geq 0\Big\}.
 \end{equation}
Finally, we can state our optimization problem.
\begin{problem}
\label{pb:first}
Find the maximizers and the maximum value of:
$$
J(c,B) = \mathbb E \left(\int_{0}^{\infty} e^{-(\rho+ \delta) t }
\left( \frac{c(t)^{1-\gamma}}{1-\gamma}
+ \delta \frac{\big(k B(t)\big)^{1-\gamma}}{1-\gamma}\right) dt
\right)
$$
over all $\pi=(c,B,\theta) \in \Pi(w,x;\phi)$.
\end{problem}


 In order for Problem \ref{pb:first} to be tractable, the parameters must respect some conditions. First we define
\begin{equation}\label{DEF_BETA}
\beta:= r+\delta - \mu_y+  \sigma_y^\top \kappa, \qquad \qquad
\beta^\phi_{\infty}:= \int_{-d}^0   e^{(r+\delta)s}  \phi (ds),
\end{equation}

\begin{assumption}\label{HYP_BETA-BETA_INFTY}
We assume the following:
\begin{itemize}
  \item[(i)]
  \begin{equation}\label{EQ_HYP_BETA-BETA_INFTY}
\beta- \beta^\phi_{\infty} >0.
\end{equation}

  \item[(ii)]
\begin{equation}\label{HYP_POSITIVITY_DEN_NU}
\rho + \delta -(1-\gamma) (r +\delta +\frac{\kappa^\top \kappa}{2\gamma }) >0,
\end{equation}
\end{itemize}
\end{assumption}
The inequality (i) is needed to prove (see \cite[Theorem 2.1]{BGPZ}) the infinite dimensional  Markovian representation of the total wealth in \eqref{eq:GammaRewritten} as a joint function of current wealth  $W(t)$ and of the present and past path of labor income,
$(y(t+s))_{s\in [-d,0]}$.
The inequality (ii) is identical to the Merton's well-posedness condition with impatience,  the rates being adjusted by $\delta$ for the presence of the time of death.

The above  Assumption  will be considered valid throughout the paper, with one important exception. The inequality \eqref{EQ_HYP_BETA-BETA_INFTY} will  be relaxed in Section \ref{sec3} because there the weight $\phi$  is more general.\\

\indent The next Remark illustrates what would change if \emph{the driving Brownian motion for labor income has a general correlation structure with the stock market}. \\
\begin{remark}\label{General-correlation}
    Given a filtered probability space, consider two independent, $n$-dimensional Brownian motions $Z$ and $Z^*$. The filtration $\mathbb{F}$ is now the augmented, natural filtration generated by the two, while $\tau_\delta$ and $\mathbb{G}$ would be defined similarly to what done in Assumption \ref{hp:tau}. The Brownian motion $Z$ still indicates the stock noise driver. In order to build a Brownian motion with general correlation structure with respect to  $Z$, consider two $n$-dimensional lower triangular  matrices  $ C_1, C_2 $ such that
     $$ C_1^TC_1 + C_2^TC_2 =I_n  $$
     in which $I_n$ is the $n$-dimensional identity matrix.
    Then, define
    $$ Z^y := C_1 Z + C_2 Z^*  $$
     It is now clear that $Z^y$ is a Brownian motion, and its covariation  with the stock market is
    $$ \langle Z^y, Z\rangle_t = t C_1 $$
    If $ {C_2} = 0_n$, then $C_1=I_n$ and we are in the perfect correlation case, which is the one explicitly presented in  the paper.  The case of perfect, negative correlation is given by $ C_1=-I_n, C_2=0$.  The null correlation case corresponds to $C_1=0_n$, and   $Z^y=Z^*$.  Another  example is  obtained by picking correlations $(\rho_1, \ldots, \rho_n)$ and $C_1$ to be the diagonal matrix with these correlations on the diagonal. Then $C_2$ is the diagonal matrix with diagonal given by $(\sqrt{1-\rho_1^2}, \ldots \sqrt{1-\rho_n^2})$.  \\
    \indent  All the  definitions and results  in this Section remain valid if, in the SDDE for $y$, $Z$ is replaced by $Z^y$. The  only important difference is that the market is now incomplete, as the labor income idiosyncratic noise $Z^*$ cannot be hedged off. There are infinitely many state price densities now. Among them however, the state price density $\xi$ in \ref{PRICE_DENSITY} corresponds to the Foellmer and Schweizer's Minimal Martingale Measure  \cite{FS}, and its theoretical properties  are useful in practice (hedging, option pricing). So, if we continue to use  $\xi$ among the possible choices, the evaluation of derivatives and  of future labor income will follow the same lines as in the case with perfect correlation \eqref{budget-constraint}.  The consequences of correlation will appear in Section 3, since  the infinite dimensional HJB equation has a more general second order term.  This  in turn affects the optimal hedging strategy,  as explicitly  calculated in Remark \ref{VICEIRA}.
\end{remark}

\subsection{Positivity of the labor income process}
\label{sec:pos}
Consider $C([-d,0]; \mathbb{R})$, the space of continuous, real valued functions on $[-d,0]$ endowed with the sup norm. The space $C([-d,0]; \mathbb{R})$ is a Banach lattice with the standard pointwise order.
The  norm dual of $C([-d,0]; \mathbb{R})$ is  $\mathcal{M}$, the space of  Radon measures on $[-d,0]$ (for further details on this duality,  see \cite{AB}).  The ordering on $\cal M$ is the natural one (see \cite[Section 8.10]{AB}), and  we briefly recall its definition. Given $\mu, \nu \in \cal M$
\begin{equation}\label{order}
\mu \leq \nu \qquad \text {iff} \qquad \int_{-d}^0 g(s) \mu(ds) \leq  \int_{-d}^0 g(s) \nu(ds)  \qquad  \forall g\geq 0,  g \in C([-d,0]; \mathbb{R})
\end{equation}
Denote by $\mathcal{M}_+$   the cone of  nonnegative measures in $\cal M$ and  by $\mathcal{M}_+^0$   the lattice subspace  $ \mathcal{M}_+ \cap \mathcal M^0$.

 The next Proposition can be proved via a variation of constants technique, exactly as  in  Proposition 2.6 of \cite{BGP}, where the measure $\phi$ is absolutely continuous with respect to the Lebesgue measure on $[-d,0]$. So, we skip the proof.

\begin{proposition}\label{pr:ypositive}
Let $y(t)=y^{x}(t; \phi)$ be the solution at time $t$ of the labor income SDDE in \eqref{DYNAMICSy}. Let \begin{eqnarray}\label{eq:E}
	E(t)&:=& e^{(\mu_y - \frac{1}{2} \sigma_y^\top \sigma_y) t + \sigma_y^\top Z(t)}\\
	I(t)&:= &\int_0^t     E^{-1}(u) \left(\int_{-d}^0y(u+s) \phi(ds)\right) du,
\label{eq:I}
\end{eqnarray}
then $y$ admits a feedback representation:
\begin{equation}\label{VARIATION_CONSTANTS_FORMULA}
y(t)= E(t)\big(x_0 + I(t)\big).
\end{equation}
 As a consequence, when $x_0>0$, $x_1 \in L^2_+([-d,0);\mathbb{R})$  and $\phi \in \mathcal{M}^0_+$, then $y^x(t)>0$ holds $P$-a.s., for all $t\ge 0$.  \end{proposition}

The last statement on the positivity of labor income process can partly be reversed in the sense that, if $\phi \not\in \mathcal{M}^0_+$, then for some positive initial datum the labor income  eventually  takes negative values.
This is the content of the following Proposition, whose proof is in Appendix \ref{app:lemmapos}.

\begin{proposition}
\label{neg}
Let $\phi$ be   in $\mathcal M^0$.  Then, $\phi \geq 0$ if and only if,  for all initial data $x_0>0$ and $x_1 \in L^2_+([-d,0);\mathbb{R})$, $y$ has positive paths $P$-a.s.
\end{proposition}

\section{The optimization problem with a fixed Radon measure $\phi$}

 We now turn to the resolution of Problem \ref{pb:first} by the dynamic programming approach. As already anticipated,
the findings here are an extension of the setup in \cite{BGP} to the case of general measures $\phi\in \mathcal{M}^0$.
Whenever their results still go through with general $\phi$ without substantial changes,   we will simply refer to \cite{BGP} for the proof. In here, the use of a Radon measure  $\phi$  makes the adjoint $A_\phi^*$ of the crucial operator $A_\phi$ (defined below in \eqref{DEF_A}) more difficult to find and handle.  This is why after introducing the mathematical setup in Subsection \ref{sec3.1}, the Subsection \ref{sec:adjoint}  is entirely dedicated  to the computation of the adjoint $A_\phi^*$. Given that, the solution of the problem will follow   as in  \cite{BGP}.
\subsection{Mathematical setup}
\label{sec3.1}
The state equation for the labor income $y$ is a SDDE, hence $y$ is non Markovian.  In order to recover Markovianity and implement dynamic programming,  the authors in \cite{BGP}
extend the state, so to include the past path\footnote{On the extension, see also  \cite{VINTER},\cite{CHOJNOWSKA-MICHALIK_1978} or the books \cite[Section 0.2]{DAPRATO_ZABCZYK_RED_BOOK}\cite[Section 2.6.8]{FABBRI_GOZZI_SWIECH_BOOK}.}.
The extended state now lives in an infinite dimensional Hilbert space, the Delfour-Mitter space $M_2$
$$
M_2 := \mathbb R \times L^2\big( [-d, 0); \mathbb R \big),
$$
with inner product, for $x=(x_0, x_1), y=(y_0, y_1) \in M_2$, defined as
\begin{equation*}
\langle x,y \rangle_{M_2}:= x_0 y_0 + \langle x_1, y_1 \rangle_{L^2([-d,0); \bR)} = x_0 y_0 + \int_{-d}^0x_1(s)y_1(s)\, {\rm d}s.
\end{equation*}
For the sake of simplicity we will drop the subscript $L^2([-d,0);\bR)$ from the inner product in  $L^2([-d,0); \bR)$, writing simply $\langle x_1, y_1\rangle$.
More information on the Delfour-Mitter space can be found e.g. in the book \cite[Part II - Chapter 4]{BENSOUSSAN_DAPRATO_DELFOUR_MITTER}). \\
\indent The measure  $\phi \in \mathcal{M}^0$ is fixed.   The Sobolev space $W^{1,2}([-d,0];\mathbb{R})$ is the space of all $f \in L^2([-d,0);\mathbb{R})$ such that the weak derivative $Df$ is also square integrable\footnote{ Note that, differently from the notation we use for $L^2$ spaces on intervals (see, footnote 2 above), for $W^{1,2}$ spaces on intervals we include the extremes on the intervals (even if usually they are not) to underline that such functions can be taken well defined and continuous up to the boundary.}  $W^{1,2}([-d,0];\mathbb{R})$ is endowed with the norm $\| f\| = \| f \|_{ L^2([-d,0);\mathbb{R})}
+ \| Df \|_{ L^2([-d,0);\mathbb{R})}$ (see e.g. \cite[Chapter 8]{Brezis} for a simple introduction to Sobolev spaces).
\\
\indent
To embed the state $y$ of the original problem
in the space $M_2$, similarly to \cite{BGP} we introduce a linear (unbounded) operator $A_\phi: \mathcal D (A_\phi) \subset M_2 \rightarrow M_2$:
\begin{eqnarray}
A_\phi(x_0,x_1) := \left(\mu_y x_0 +\int_{-d}^0x_1(s)\, \phi(ds) ,\ { \frac{\partial}{\partial s} x_1}   \right), \label{DEF_A} \\
\mathcal D(A_\phi) :=\left\{(x_0,x_1) \in M_2: x_1(\cdot) \in
W^{1,2}\left( [-d, 0]; \mathbb R \right), x_0 = x_1(0)\right\} \nonumber,
\end{eqnarray}
\noindent where $\frac{\partial}{\partial s}x_1$ is the weak derivative on $[-d, 0)$ of $x_1$ while
$\mu_y,\phi$ are the constants appearing in the drift of the labor income dynamics equation. Consider also
 the linear (bounded) operator $C:M_2 \rightarrow  \mathbb R^n \times L^2([-d,0); \mathbb R)$ defined as
\begin{equation*}
C(x_0,x_1):= \left(x_0  \sigma_y , 0\right),
\end{equation*}
 where $\sigma_y $ is the diffusion coefficient of the labor income and $0$ is the null function in $L^2([-d,0); \mathbb R)$.
 Following e.g. \cite[Section 2.6.8]{FABBRI_GOZZI_SWIECH_BOOK}
the state equation \eqref{DYNAMICSy} can be formally rewritten  as follows. Given a solution $y$ to \eqref{DYNAMICSy},
for $t \ge 0$ and $s\in [-d,0)$ set
$X_0(t)=y(t)$ and $X_1(t)(s)=y(t+s)$
(so here $X_1:\R_+\to L^2([-d,0);\R)$.
Assume for the moment that $(X_0(t), X_1(t))\in\mathcal D(A_\phi)$ for all $t\ge 0$. Then, from \eqref{DYNAMICSy} we get
 $$
 \left \{
 \begin{split}
    d X_0(t) =&  \left(\mu_y X_0(t) +\int_{-d}^0X_1(t)(s)\, \phi(ds)\right ) dt + X_0(t)  \sigma_y^\top dZ(t) \\
    d X_1(t)(s) = & \, \frac{\partial}{\partial s} X_1(t)(s) dt
 \end{split} \right.
 $$
which can be written, using $A_\phi$, as
\begin{equation}\label{EQ:dX}
d X(t) = A_\phi X(t) dt + \big(C X(t) \big)^\top dZ_t
\end{equation}
Now, Proposition A.27 in \cite{DAPRATO_ZABCZYK_RED_BOOK} shows that $A_\phi$ generates a strongly continuous semigroup in $M_2$.
The findings in \cite{Flandoli90SICON} then ensure that for every initial datum $x=(x_0,x_1)\in M_2$ the equation  \eqref{EQ:dX} admits a unique mild solution
(see e.g.
\cite[Section 1.4.1]{FABBRI_GOZZI_SWIECH_BOOK} for the precise definition of mild solution) $X^{x}(\cdot;\phi)=(X_0^{x}(\cdot;\phi),X_1^{x}(\cdot;\phi))$.
Such solution  can be identified,
thanks to Theorem 3.9 and Remark 3.7 in \cite{CHOJNOWSKA-MICHALIK_1978}
with the solution $y^x(\cdot;\phi)$ of the  labor income equation \eqref{DYNAMICSy} in the sense that
$$
X^x(t;\phi)= \big(y^x(t;\phi), y^x (t+s;\phi)_{\mid s \in [-d,0)} \big),
\qquad \forall t \ge 0.
$$
Therefore, the complete state equations system  becomes
\begin{eqnarray}\label{DYN_W_X_INFINITE_RETIREMENT_II}
\begin{split}
\left\{\begin{array}{l}
dW(t) =  \left[ (r+\delta) W(t)+ \theta^\top(t) (\mu-r \mathbf 1)  +  X_0(t) - c(t) - \delta B(t) \right] dt  +  \theta^\top (t) \sigma dZ(t),\\
  dX(t) = A_\phi X(t) dt + \big(C X(t) \big)^\top dZ_t\\
W(0)= w,\\
X_0(0) = x_0,\quad \quad
 X_1(s) = x_1(s) \mbox{ for $s \in  [-d,0)$}
 \end{array}\right. \end{split}
\end{eqnarray}
Rewriting the state equations in this form is necessary to write  the associated infinite dimensional Hamilton-Jacobi-Bellman equation which allows to find the value function and optimal controls.\\
\indent The budget constraint \eqref{budget-constraint} can also  be written as a Markovian function of the state variables $(W, X)$. In fact,  in  \cite[Theorem 2.1]{BGPZ} the authors show that,
given any $x\in M_2$ and any Radon measure $\phi$ over $[-d,0]$,
\begin{equation}\label{eq:rewritingconstraint}
\xi^{-1}(t)\mathbb E\left( \int_t^{+\infty} \xi(u) y^x(u;\phi) du \Bigg\vert \mathcal F_t\right)
=  g^\phi_{\infty}X_0^x(t;\phi)+ \langle h^\phi_{\infty}, X_1^x(t;\phi) \rangle,
\qquad \forall t \ge 0,
\end{equation}
where the constant $g^\phi_{\infty}\in \mathbb R$ and the function
$h^\phi_{\infty}: [-d,0]\longrightarrow \mathbb R_+$ are defined as follows:
 \begin{equation}
 \label{DEF_g_infty_h_infty}
g^\phi_{\infty} := \dfrac{1}{\beta- \beta^\phi_{\infty} },
\qquad
h^\phi_{\infty} (s): = g^\phi_{\infty}
{ \displaystyle 
	\int_{-d}^s   e^{-(r+\delta)(s-\tau)} \phi (d\tau) ,}
\end{equation}
with $\beta$ and $\beta^\phi_\infty$ given by \eqref{DEF_BETA}.
 Therefore, for $(w,x)\in \R\times M_2$, $\phi\in \mathcal{M}^0$, $\pi\in \Pi^0$, the total wealth from \eqref{budget-constraint} can be written as
\begin{equation}\label{eq:GammaRewritten}
\Gamma^{w,x}(t;\phi,\pi)
= W^{w,x}(t;\phi,\pi) + g^\phi_{\infty}X_0^x(t;\phi)+
\langle h^\phi_{\infty}, X_1^x(t;\phi) \rangle \
\text{ for all }  t\geq 0.
\end{equation}
Setting $\mathcal{H}:= \mathbb{R} \times M_2$, if we define the linear function $G^{\phi}: \calh\to \R$ as
\begin{equation}\label{eq:Gdef}
G^\phi(w,x_0,x_1): = w + g^\phi_{\infty}x_0+
\langle h^\phi_{\infty}, x_1 \rangle,
\end{equation}
the total wealth becomes
\begin{equation}\label{eq:GammaG}
\Gamma^{w,x}(t;\phi,\pi)
=G^\phi(W^{w,x}(t;\phi,\pi),X_0^x(t;\phi),X_1^x(t;\phi)).
\end{equation}
and the  set $\Pi(w,x;\phi)$ of admissible controls from \eqref{DEF_PI_FIRST_DEFINITION} can be  rewritten as
\begin{equation}\label{DEF_PI_FIRST_DEFINITIONbis}
\Pi\left(w,x;\phi\right):= \Big\{ \pi \in \Pi^0
\mid
G^\phi(W^{w,x}(t;\phi,\pi),X_0^x(t;\phi),X_1^x(t;\phi))
 \geq 0\,\quad \forall t \geq 0\Big\}.
\end{equation}

Notice that, when $t=0$, the initial datum $(w,x)$ must belong to the half space
\begin{equation*}
\mathcal{H}^{\phi}_+:=\{(w,x) \in \mathcal{H} \ : G^\phi(w,x_0,x_1) \ge 0\}.
\end{equation*}
 Hence, the state constraint \eqref{budget-constraint} means that the state trajectory $(W(t),X(t))$ must remain in $\mathcal{H}^{\phi}_+$ at all times.
 Finally, set
\begin{equation*}
\mathcal{H}^{\phi}_{++}:=\{(w,x) \in \mathcal{H} \ : G^\phi(w,x_0,x_1) >0\}= \mathrm{Int} ( \mathcal{H}^{\phi}_+),
\end{equation*}
in which $\mathrm{Int(A)}$ denotes the interior part of a set $A$.

\subsection{The adjoint $A_\phi^*$}
\label{sec:adjoint}

 To solve the infinite dimensional HJB equation associated to our problem, we need an explicit representation of the adjoint operator $A_{\phi}^*$ of $A{_\phi}$.
In exhibiting $A_{\phi}^*$,
first of all  define the linear bounded operator $\bar L$ as follows:
\begin{equation}\label{eq:Lbardef}
\bar L:C([-d,0];\R)\to L^2([-d,0);\R), \qquad
(\bar Lz)(s)=\int_{-d}^{s}z(u-s)\phi(du)
\end{equation}
 Then, using a result of
\cite[Ch.4,\S 4.4]{BENSOUSSAN_DAPRATO_DELFOUR_MITTER}
(generalized in Appendix \ref{app:wellposed},
\gre{Lemma \ref{newlemma.},}
to the case when $\phi$ is a stochastic process),
one extends $\bar L$ to a linear bounded operator
$$
\bar L:L^2([-d,0);\R)\to L^2([-d,0);\R),
$$
still denoted in the same way. Then, use it to define the linear operator $F$:
\begin{equation}\label{eq:Fdef}
 F:M_2\to M_2;
\qquad F(x_0,x_1)=\left(x_0,\bar L x_1\right)
\end{equation}
Such $F$ is the so-called ``structural operator'' associated to the operator $A_{\phi}$.
 From \cite[Theorem 4.6, p. 269]{BENSOUSSAN_DAPRATO_DELFOUR_MITTER} one derives the adjoint as follows.
\begin{proposition}\label{pr:adjoint}
The adjoint  $A_{\phi}^*$ is a linear operator
$A_{\phi}^*: \mathcal D(A_{\phi}^*)\subset M_2 \longrightarrow M_2$
with domain
\begin{align}\label{EQ_DOMAINADJOINT_OPERATOR_A*}
\mathcal D(A_{\phi}^*):=\bigg\{& z \in M_2:\;\exists \xi \in D(A_\phi) \; and \;
\zeta \in W^{1,2}([-d, 0]; \mathbb R ),\, \zeta(-d) = 0,
\\
&   z=F\xi+(0,\zeta)\;\bigg\}
\end{align}
and
\begin{equation}\label{EQ_ADJOINT_OPERATOR_A*}
A_{\phi}^*z:= FA_{\phi}\xi + (\zeta(0),-\zeta')
=
\left(\mu_y \xi_0 +\int_{-d}^0\xi_1(s)\, \phi(ds)+\zeta(0),
(\bar L\xi_1') -  \zeta^{\prime}  \right)
\end{equation}
\end{proposition}
Recall that the couple
$(g_{\infty}, h_{\infty}) \in M_2 $  is:
 \begin{eqnarray}\label{DEF_g_infty_h_infty}
 \left\{
\begin{array}{ll}
g_{\infty} &:= \dfrac{1}{\beta- \beta_{\infty} },\\
\\
h_{\infty} (s)&: = g_{\infty}    { \displaystyle 
	\int_{-d}^s   e^{-(r+\delta)(s-\tau)} \phi (d \tau),}
\end{array}
\right.
\end{eqnarray}
with $\beta $ and $\beta_{\infty}$ defined in (\ref{DEF_BETA}).
\begin{lemma}\label{LEMMA_DIFF_EQ_G_INFTY_H_INFTY}
The couple  $(g_{\infty},h_{\infty}) \in \mathcal D(A_{\phi}^*)$ and
\begin{equation}\label{eq:A*gh}
A_{\phi}^*(g_{\infty},h_{\infty})=
\left(\mu_y g_\infty+g_\infty \int_{-d}^{0}e^{(r+\delta)s}\phi(ds),
(r+\delta)h_\infty
\right).
\end{equation}
\end{lemma}
\begin{proof}
By Proposition \ref{pr:adjoint},
we need to find
$\xi \in \cald(A_{\phi})$ and $\zeta \in W^{1,2}([-d, 0]; \mathbb R )$ with $\zeta(-d) = 0$,
such that
\begin{equation}\label{eq:xizeta}
(g_\infty,h_\infty)=F\xi+(0,\zeta).
\end{equation}
This implies immediately that it must be $g_\infty=(F\xi)_0=\xi_0$.
and, by simple computations, that the couple
$$
\xi_1(u):=g_\infty e^{(r + \delta)u}, \qquad \zeta(u) =0, \qquad \forall u\in [-d,0]
$$
satisfies \eqref{eq:xizeta}. In particular,
$$
(g_{\infty},h_{\infty}) \in F(\cald(A_{\phi})) \subseteq \mathcal D(A_{\phi}^*).
$$
Then, by \eqref{EQ_ADJOINT_OPERATOR_A*},
\begin{equation}\label{eq:A*ghbis}
A_{\phi}^*(g_{\infty},h_{\infty})= FA_{\phi}\xi=
\left(\mu_y \xi_0+\int_{-d}^{0}\xi_1(s)\phi(ds),\bar L (A_{\phi}\xi)_1\right).
\end{equation}
Since
$$
(A_{\phi}\xi)_1(u)=\xi_1^\prime=g_\infty(r+\delta) e^{(r + \delta)u}
$$
we get, by \eqref{eq:Lbardef},
$$
\bar L (A_{\phi}\xi)_1(\cdot)=g_\infty(r+\delta) \int_{-d}^{\cdot} e^{(r + \delta)(u-\cdot)}\phi(du)
=(r+\delta) h_\infty(\cdot)
$$
which concludes the proof.
\end{proof}

\subsection{The HJB equation and its explicit solution}
\label{SSE:HJBEXPLICIT}

\begin{notation}
Let $p= (p_1,p_2)$
be a generic vector of $\calh=\mathbb R \times M_2$,
and let $S(2)$ denote the space of real symmetric matrices of dimension $2$, and
$P$ an element of $S(2)$, with
\begin{equation*}\label{NOTATIONS_P}
P=\left(\begin{array}{cc}
P_{11}&  P_{12}  \\
 P_{21}  &  P_{22}
\end{array} \right).
\end{equation*}
For any given function $u: \calh \longrightarrow \mathbb R$,
we denote by $Du = \left(u_w,u_{x}\right)= \left(u_w,(u_{x_0},u_{x_1})\right)\in \calh$
its gradient and by
$$D^2_{wx_0} u = \left(\begin{array}{cc}
 u_{ww}&  u_{wx_0}  \\
 u_{x_0 w}  &  u_{x_0 x_0}
\end{array} \right)\in S(2)
$$
its second derivatives with respect to the first two components $(w,x_0)$,
whenever they exist and the mixed derivatives coincide.
\end{notation}

The Hamiltonian
$\mathbb H: \mathbb R \times M_2 \times  (\mathbb R \times \cald (A_{\phi}^*))   \times S(2) \longrightarrow  [-\infty, \infty]$  is defined as follows
\begin{align}\label{DEF_HAMILTONIAN_INF_RETnew}
&\mathbb H(w,x,p,P)  :=
\\
\notag
&\sup_{(c,B,\theta)\in \R_+\times \R_+\times \mathbb R^n}
\left\{\left \langle \Delta(w,x,\theta,c,B),p \right \rangle_{\R\times M_2}+ \frac12 Tr \left (\Sigma(\theta,x_0) P \Sigma^*(\theta,x_0) \right ) + U(c,B)\right\}
\end{align}
where $\Delta(w,x,\theta,c,B)$ and $\Sigma(\theta,x_0)$  are the drift and the (reduced) diffusion coefficients  of the infinite dimensional system  \eqref{DYN_W_X_INFINITE_RETIREMENT_II}, while $U$
is the instantaneous utility function in \eqref{gaininfty}.
The HJB equation  associated with the optimization Problem \ref{pb:first} is
\begin{equation} \label{eq:HJB1}
(\rho+\delta) v = \mathbb H\left(w,x,D v, D^2_{wx_0} v\right),
\end{equation}
To compute the Hamiltonian we separate the part depending on the controls
from the rest, which can be taken out of the supremum.  \\

When $x\in D(A_\phi)$, we have
\begin{align}\label{DEF_HAMILTONIAN_INF_RET}
\mathbb H(w,x,p,P)  :=
\mathbb H_1(w,x,p,P_{22}) + \mathbb H_{max}(x_0,p_1,P_{11},P_{12}),
\end{align}
where
\begin{equation}\label{DEF_H_1_INF_RET}
\mathbb H_1(w,x,p,P_{22}) :=  (r+\delta) w p_1+x_0 p_1 + \langle A_{\phi}x,p_2  \rangle_{M_2}+
\frac{1}{2}   \sigma_y^\top  \sigma_y x_0^2 P_{22},
\end{equation}
and
\begin{equation}\label{DEF_H_MAX_INF_RET}
 \mathbb H_{max}(x_0,p_1,P_{11},P_{12}) := \sup_{(c,B,\theta)\in \R_+\times \R_+ \times \mathbb R^n} \mathbb H_{cv}(x_0,p_1,P_{11},P_{12}; c,B,\theta)
\end{equation}

\begin{align}\label{DEF_H_cv_INF_RET}
&\mathbb H_{cv}
(x_0,p_1,P_{11},P_{12}; c,B,\theta):=
\\
&=\frac{ c^{1-\gamma}}{1-\gamma}
+ \frac{\delta \big(k B\big)^{1-\gamma}}{1-\gamma}
+[\theta^\top(\mu-r\mathbf 1) - c-\delta B]p_1
+\frac{1}{2} \theta^\top  \sigma  \sigma^\top   \theta P_{11}  +  \theta^\top  \sigma \sigma_y x_0   P_{12}
\nonumber
\\
\nonumber
= &
\frac{ c^{1-\gamma}}{1-\gamma}- c p_1
+ \frac{\delta \big(k B\big)^{1-\gamma}}{1-\gamma} -\delta B p_1
+\theta^\top(\mu-r\mathbf 1)p_1
+\frac{1}{2} \theta^\top  \sigma  \sigma^\top   \theta P_{11}  +  \theta^\top  \sigma \sigma_y x_0   P_{12}.
\end{align}

Now note that, thanks to the last equality above, whenever $p_1>0$ and $P_{11}<0$,
the maximum in (\ref{DEF_H_MAX_INF_RET}) is achieved at
\begin{eqnarray}\label{MAX_POINTS_HAMILTONIAN__INF_RET}
\left\{
\begin{split}
c^{*} &:=  
p_1^{-\frac{1}{\gamma}}  ,  \\
B^{*}&:=k^{ -b} p_1^{-\frac{1}{\gamma}}   , \\
\theta^{*} &:=   - (\sigma\sigma^\top)^{-1} \frac{(\mu-r\mathbf 1) p_1 + \sigma  \sigma_y x_0  P_{12}}
{P_{11}},
\end{split}
\right.
\end{eqnarray}
where
\begin{equation}\label{DEB_b}
b= 1-\frac{1}{\gamma}.
\end{equation}
Hence, for $p_1>0$ and $P_{11}<0$ we have, by simple computations,
\begin{eqnarray}\label{EQ_EXPL_HAMILTONIAN_INF_RET}
\begin{split}
&\mathbb H(w,x,p,P)=
[(r+\delta) w +x_0 ]p_1+\langle A_{\phi}x,p_2  \rangle_{M_2}
+\frac{\gamma}{1-\gamma}  p_1^{b}
\big(
1+\delta k^{-b}  \big)
+\frac{1}{2}   \sigma_y^\top  \sigma_y x_0^2 P_{22}\\
&\qquad -\frac{1}{2P_{11}}
\left[(\mu-r \mathbf 1)  p_1+ \sigma \sigma_y x_0 P_{12}\right]^\top
(\sigma\sigma^\top)^{-1}
\left[(\mu-r \mathbf 1) p_1+ \sigma \sigma_y  x_0 P_{12}\right].
\end{split}
\end{eqnarray}
Therefore, if the unknown $v$ satisfies $v_w>0$ and $v_{ww}< 0$,
the HJB equation in (\ref{eq:HJB1}) reads
\begin{eqnarray}\label{HJB_CLEAN_INF_RET}
\begin{split}
(\rho+\delta) v&=
 [(r+\delta) w +x_0]  v_w +\langle A_{\phi}x,v_x  \rangle_{M_2}+
\frac{\gamma}{1-\gamma}  v_w^{b} \big(1 
+\delta k^{-b}  \big)
+\frac{1}{2}   \sigma_y^\top  \sigma_y x_0^2 v_{x_0x_0}\\
&-\frac{1}{2v_{ww}}
\left[(\mu-r \mathbf 1)  v_w+ \sigma \sigma_y x_0 v_{w x_0}\right]^\top
(\sigma\sigma^\top)^{-1}
\left[(\mu-r \mathbf 1) v_w+ \sigma \sigma_y  x_0 v_{w x_0}\right].
\end{split}
\end{eqnarray}
On the other hand we must also note that, when $p_1<0$ or $P_{11}>0$,
the Hamiltonian $\mathbb H$ is $+\infty$, while, when $p_1P_{11}=0$,
different cases may arise depending on $\gamma$ and on the sign of other terms.\\

\medskip

The Hamiltonian  specification above, and hence the HJB equation in
\eqref{HJB_CLEAN_INF_RET}, makes sense only for $x\in D(A_{\phi})$. In order to write the HJB equation for more general states as those obtained from an SDE evolution, we use the adjoint $A^*_\phi$.  In fact, if  $p_2$ in \eqref{EQ_EXPL_HAMILTONIAN_INF_RET}
(or $v_x$ in \eqref{HJB_CLEAN_INF_RET}) belong to $D(A_{\phi}^*)$, the adjoint properties  imply that the Hamiltonian and the  HJB equation become
\begin{eqnarray}\label{EQ_EXPL_HAMILTONIAN_INF_RETbis}
\begin{split}
\mathbb H(w,x,p,P)=
&[(r+\delta) w+x_0] p_1+\langle x,A_{\phi}^*p_2  \rangle_{M_2}
+\frac{\gamma}{1-\gamma}  p_1^{b} \big(
1+\delta k^{-b}  \big)\\
&+\frac{1}{2}   \sigma_y^\top  \sigma_y x_0^2 P_{22}\\
&-\frac{1}{2P_{11}}
\left[(\mu-r \mathbf 1)  p_1+ \sigma \sigma_y x_0 P_{12}\right]^\top
(\sigma\sigma^\top)^{-1}
\left[(\mu-r \mathbf 1) p_1+ \sigma \sigma_y  x_0 P_{12}\right].
\end{split}
\end{eqnarray}
\begin{eqnarray}\label{HJB_CLEAN_INF_RETbis}
\begin{split}
(\rho+\delta) v=
&  [(r+\delta) w+x_0]  v_w +\langle x,A_{\phi}^*v_x  \rangle_{M_2}+
\frac{\gamma}{1-\gamma}  v_w^{b} \big(1 
+\delta k^{-b}  \big)\\
&+\frac{1}{2}   \sigma_y^\top  \sigma_y x_0^2 v_{x_0x_0}\\
&-\frac{1}{2v_{ww}}
\left[(\mu-r \mathbf 1)  v_w+ \sigma \sigma_y x_0 v_{w x_0}\right]^\top
(\sigma\sigma^\top)^{-1}
\left[(\mu-r \mathbf 1) v_w+ \sigma \sigma_y  x_0 v_{w x_0}\right].
\end{split}
\end{eqnarray}

The next definition  is the same as in  \cite{BGP}.
\begin{definition}\label{DEF_SUPERSOLUTION_INF_RET}
A function $ u : \calh^{\phi}_{++} \longrightarrow \mathbb R $
is a \emph{classical solution}
of the HJB equation (\ref{eq:HJB1}) in $\calh^{\phi}_{++}$
if the following holds:
\begin{itemize}
  \item[(i)] $u$ is continuously Fr\'{e}chet differentiable in $\calh^{\phi}_{++}$ and  admits
      continuous second derivatives with respect to $(w,x_0)$ in $\calh^{\phi}_{++}$;
  \item[(ii)] $u_x(w,x) \in \cald(A_{\phi}^*)$ for every $(w,x)\in \calh^{\phi}_{++}$
  and $A^*_{\phi}u_x$ is continuous in $\calh^{\phi}_{++}$;
  \item[(iii)]
for all $(w,x) \in \calh^{\phi}_{++} $ we have
\begin{equation}\label{EQ_SOLUTION_INF_RET}
\begin{split}
(\rho + \delta) u -
\mathbb H \big(w,x,Du, D^2_{wx_0} u\big)= 0.
\end{split}
\end{equation}
\end{itemize}
\end{definition}



\begin{proposition}\label{PROP_COMPARISON_FINITENESS_VAL_FUN_INF_RET}
Define, for $(w,x) \in \calh^{\phi}_{++}$,
\begin{equation}\label{EQ_GUESS_BAR V_INF_RET}
\bar v(w,x):=
\frac{f_{\infty}^{\gamma}(G^{\phi}(w,x))^{1-\gamma}}{1-\gamma},
\end{equation}
with $f_{\infty}>0$
defined as
\begin{equation}\label{EXPRESSIONS_f_INF_RET}
 f_{\infty}:= (1 + \delta k^{-b}) \nu,
 \end{equation}
where
\begin{equation}\label{DEF_nu}
\nu  := \frac{\gamma}{\rho + \delta -(1-\gamma) (r + \delta +\frac{\kappa^\top \kappa}{2\gamma })}>0,
\end{equation}
$G^{\phi}$ defined in \eqref{eq:Gdef},
and $b$ as in (\ref{DEB_b}).
Then, $\bar v$ is a classical solution of the HJB equation
(\ref{eq:HJB1}) on $\calh^\phi_{++}$.

\end{proposition}
\begin{proof}
The operator $A^*_{\phi}$ acts on $(g_{\infty}, h_{\infty})$ in the same way as the operator $A^*$ considered in  \cite[Lemma 3.2]{BGP}). Mutatis mutandis, the proof follows the same lines of the cited reference.
 \end{proof}

\begin{remark}\label{rm:vonboundary}
The function $\bar v$ can be defined also in $\calh_+$ by setting,
on  its frontier  $\partial\calh_+=\{G^{\phi}(w,x)=0\}$,
$$\bar v(w,x)=0, \quad \hbox{when $\gamma\in (0,1)$}:
\qquad and \qquad
\bar v(w,x)=-\infty, \quad \hbox{when $\gamma\in (1,+\infty)$}.
$$
From now on we will consider $\bar v$ defined on $\calh_+$.
\end{remark}

\begin{remark}\label{rm:onregofHJB}
 Observe that for infinite dimensional HJB equations like
\eqref{EQ_SOLUTION_INF_RET} there are no available results on existence/uniqueness of classical solutions.
 On one side, it could be feasible to adapt to this case known results on existence/uniqueness of viscosity solutions like the ones of \cite[Chapter 3]{FABBRI_GOZZI_SWIECH_BOOK} and, in the so-called path-dependent PDEs setting, of
\cite{CFGRT,RenRosest}. On the other side, regularity results are far from  being  available. To our knowledge, the only regularity result which applies to a similar family of second order HJB equations is the one of \cite{RosestolatoSwiech17}
which only proves partial regularity, i.e. the derivative $\frac{\partial}{\partial x_0}$ (in the so-called ''present'' direction) is well defined and continuous.
\end{remark}

\subsection{The optimal controls}

The following result provides the solution of the
optimization problem without robustness.

\begin{theorem}\label{TEO_MAIN_INF_RET}
 Let $(w,x)\in\mathcal{H}^\phi_+$.
The value function $V$ equals
	\begin{equation}
\label{barV}
	 V(w,x;\phi) =   \frac{ f_{\infty}^{\gamma} \left(
		G^{\phi}(w,x)\right)^{1-\gamma} }
	{1-\gamma},
	\end{equation}
where the constant $ f_{\infty}>0$ is defined as
$$ f_{\infty} = (1+\delta k^{-b})\nu, $$
where
$$ b= 1-\frac{1}{\gamma}  \ \ ; \   \  \nu = \frac{\gamma}{\rho + \delta -(1-\gamma)(r+\delta + \frac{\kappa^T\kappa}{2\gamma}) } >0.$$
 The optimal total wealth, starting at $(w,x)$ is given by
\begin{equation}\label{Gamma-infty2}
(\Gamma^{w,x})^*(t;\phi): =
 G^\phi \left((W^{w,x})^*(t;\phi),X^x(t;\phi)\right)
 \end{equation}
where $\left((W^{w,x})^*(t;\phi),X^x(t;\phi)\right)$ are the solutions of the system (\ref{DYN_W_X_INFINITE_RETIREMENT_II}), starting at $(w,x)$ and with controls defined in feedback form (below we write $\Gamma^*(t)$ for $(\Gamma^{w,x})^*(t;\phi)$):
	\begin{align}\label{OPTIMAL_STRATEGIES_RET_INF}
	\begin{split}
	c^{*}(t)&:=
	f_{\infty}^{-1}  \Gamma^*(t)   \\
	B^{*}(t)&:=k^{ -b } f_{\infty}^{-1}  \Gamma^*(t)      \\
	\theta^{*}(t)&:=   (\sigma\sigma^\top)^{-1}(\mu-r \mathbf 1)  \frac{  \Gamma^*(t)}{ \gamma} -g^{\phi}_{\infty}X^x_0(t;\phi)(\sigma^\top)^{-1} \sigma_y ,
	\end{split}
	\end{align}
As a consequence,  $\Gamma^*$ is a Dol\'{e}ans exponential with dynamics
	\begin{align}\label{DYN_GAMMA*_PB1}
	\begin{split}
	d  \Gamma^* (t) =& \Gamma^* (t) \Big(  r + \delta +\frac{\kappa^\top \kappa}{\gamma}
	- f_{\infty}^{-1}\big( 1 
	+\delta k^{-b}\big) \Big)dt
	+  \frac{\Gamma^* (t)}{\gamma } \kappa^\top dZ(t).
	\end{split}
	\end{align}
and initial condition $\Gamma^*(0) = w + g^{\phi}_{\infty}x_0 + \langle h^{\phi}_{\infty}, x_1 \rangle$.
\end{theorem}

\begin{proof}
 The proof is long and non trivial but it can be done in the same way as in \cite{BGP}. It consists in the following main steps for $\gamma\in (0,1)$.
\begin{itemize}
\item Show that the set of admissible strategies when the initial point belongs to the boundary of $\mathcal{H}^\phi_+$ is  made only by one element, which keeps the state on the boundary forever. This is a key issue in the state constrained problem.
\item Prove the fundamental identity which, in turns, implies $V(w,x;\phi)\le \bar v(w,x)$, for every $(w, x) \in \mathcal{H}^\phi_+$,
\item Show that $\bar v=V$ (Verification Theorem) and find the optimal strategies in feedback form as the maximizers of the Hamiltonian.
\end{itemize}
When $\gamma >1$ the first step is the same while the other two must be done differently, using the homogeneity of the problem and the Dynamic Programming Principle.

 The difference between the present case and the one treated in \cite{BGP} is the fact that $\phi$ is now a Radon measure.
Once a general existence and uniqueness theorem for the equation
\eqref{DYNAMICSy} is established (see Appendix \ref{app:wellposed})
the only change which arises in the present proof is the form of the adjoint operator $A^{*}_{\phi}$.
However, in all the steps described above in \cite{BGP} the only property of $A^{*}_{\phi}$ which is used is Lemma \ref{LEMMA_DIFF_EQ_G_INFTY_H_INFTY}.
Once this is established, all the technical details of the steps described above can be carried on exactly as in \cite{BGP}.
\end{proof}
 To conclude, in the following Remark we compare the structure of the optimal control with the ones in the classic Merton problem. In particular, we focus on the hedging demand $\theta^*$ and illustrate how the result would change if the driving Brownian motion in the equation for the labor income is not perfectly correlated with $Z$.
\begin{remark}\label{VICEIRA}
     The structure of the above solutions is in line with Merton's results. The value function is proportional to the utility of (running) total wealth. The optimal controls are a constant fraction of the optimal total wealth $\Gamma^*$, modulo a correction for $\theta^*$. Such correction is due to the negative hedging demand arising in the presence of  an income perfectly correlated with the market noise.  This is due to the fact that the agent is already exposed to labor income risk, identical to the stock market one, and therefore invests less than in the classic case without labor.  This is in line  with \cite{DYBVIG_LIU_JET_2010}, where there is perfect correlation  but  no path dependency. \\
    \indent However, things change in the general correlation case as introduced in Remark \ref{General-correlation}. If we let $X_0=y$ be driven by $Z^y = C_1 Z +C_2 Z^*$ instead of $Z$, then the mixed second derivative $v_{wx_0}$  in \eqref{DEF_H_cv_INF_RET} (corresponding to $P_{12}$) has the general coefficient
    $$ x_0 \theta^T \sigma C_1^T \sigma_y  $$
    instead of $ x_0\theta^T \sigma \sigma_y $, since $C_1$ is not the identity matrix anymore. Nothing else would change in the computations, and so the resulting  $\theta^*$ in \eqref{OPTIMAL_STRATEGIES_RET_INF} is
    $$ \theta^{*}(t)=   (\sigma\sigma^\top)^{-1}(\mu-r \mathbf 1)  \frac{  \Gamma^*(t)}{ \gamma} - g^{\phi}_{\infty}X^x_0(t;\phi) (\sigma^\top)^{-1}C_1^T \sigma_y. $$
     We observe that in a perfectly negative correlation case,   $C_1=-I$ and then the hedging demand increases in all the stocks. When stocks and labor income have $0$ correlation,  $C_1=0_n$, then the correction term vanishes. In a general case, the hedging demand will increase in the negatively correlated components, and decrease in the others. This extends the results found by  Viceira   in discrete time, see  the seminal paper \cite{VICERIA}.
 \end{remark}

\section{The robust problem}
\label{sec3}
\subsection{ The controls of the malevolent Nature }

 The delay measure  $\phi$ is  now  allowed to be a  (measure valued) stochastic process $\phi(\cdot)$. The process $\phi(\cdot)$ is not revealed to the agent, but is picked by an adversary player (the malevolent Nature) from a set of admissible controls  which take values in a suitable set $K$. The agent then aims at finding an optimal strategy which is  robust with respect to the Nature's move.  We start with the precise assumptions on $K$,  and then we focus on  the Nature controls $\phi(\cdot)$.
\begin{assumption}\label{asK}
 \begin{enumerate}
 \item The uncertainty set $K $ is a subset of $ \mathcal M^0$.
  \item  $K$  has an order minimum
\begin{equation*}
  \exists\, {\nu} \in K \text{ s.t. } \qquad {\nu} \, \leq \phi \qquad \forall \phi \in K,
\end{equation*}
 The order must be intended  in the natural lattice structure of the Radon measures, as recalled in \eqref{order}.
\end{enumerate}
\end{assumption}

The assumptions on the uncertainty set $K$  are inline with part of the current literature, see   \cite{BP} for more details and references.   Existence of an order  infimum in the set $K$  is  needed  since  in the resolution of the  Nature-agent game    we apply  a  monotonicity argument.  This should be contrasted with  another branch of the  literature, which bases the resolution of the maxmin problem   on topological continuity and convexity properties of the functionals and sets involved (compactness). In fact, a technical  tool which is typically used there is Sion's Minimax Theorem (see for example \cite{NN}).\\

\begin{example}
Fix $\phi_0,\psi \in  \mathcal{M}^0$  with $\psi \geq 0$. Then, the tubular neighborhood of  $\phi_0 \in \mathcal{M}^0$:
\begin{equation*}
K:=\{\phi \in  {\mathcal{M}^0}:  \phi_0-\psi \le \phi \le \phi_0+\psi\}
\end{equation*}
is a set of measures verifying the conditions   stated in the previous Assumption.
   The interpretation is that $\phi_0$ is an estimate of  the  impact of the past on   income dynamics, while $\psi $ is the   estimation error. Here the order minimum  ${\nu}$ is $\phi_0-\psi $.
\end{example}

 Since we are considering pre-death versions of the processes involved, in the sequel the filtration on $\Omega$ can be taken to be $\mathbb{F}$ as remarked in equation \eqref{eq:GFpred}.
Then, on $\Omega\times [0,\infty)$ we consider the progressive $\sigma$-algebra $\textsf{Prog}$, which is the one generated by all progressively measurable processes on $\Omega\times [0,\infty)$.
Now we recall that a ($\mathcal{M}^0$) transition kernel\footnote{See e.g. \cite[p.19]{Kallenberg} for the definition.}
$\psi$ between the spaces $(\Omega\times [0,\infty), \textsf{Prog} )$
and $([-d,0), \mathcal{B}([-d,0))$ is a mapping
$$
\psi: \left(\Omega\times [0,\infty)\right)\times
\mathcal{B}([-d,0))   \rightarrow \R$$
such that: the process $(\omega,t) \to \psi(\omega,t) (B)$ is progressively measurable for each Borel set $B$ in $[-d,0)$; and $\psi(\omega,t) (\cdot)$ is a measure in $\mathcal{M}^0$.
Clearly such transition kernel can be written as a measurable map
\footnote{
 Here we endow $\mathcal{M}^0$ with the topology inherited from the one on $\mathcal{M}$ as a dual space of $C([-d,0];\R)$, and we consider the asociated Borel $\sigma$-field.}
$$
\psi: \Omega\times [0,\infty)\rightarrow \mathcal{M}^0,
$$
i.e.  \emph{such Radon transition kernel can be seen as a process taking values in the set $\mathcal{M}^0$.}
The Nature's controls will be the transition  kernels as above, which are in addition  $K$-valued.

\begin{definition}
\label{df:adm}
A  $K$-valued transition kernel $\psi$ is said to be \emph{admissible} if
it  is locally bounded in time for the total variation norm. That is, for all $T>0$ there exists a constant $c_T\geq 0$ such that
          $$  \sup_{0\leq t\leq T} \| \psi(\omega, t) (ds) \| \leq c_T \qquad \P-a.s. $$
\end{definition}
\begin{definition}
\label{df:controlNature}
The set of Nature's controls  $\mathcal{K}$ is the set of all  admissible transition kernels. The  controls will be denoted by $\phi(\cdot)$. When $B$ is fixed, when considering the r.v. $\phi(t, \cdot)(B)$ we will write $\phi(t)(B)$ as a shortcut. In the same way, the (random) measure $\phi(t, \cdot)(ds)$ will be denoted by $ \phi(t)(ds)$.
\end{definition}
 Consider now a map
$$
N:  \Omega \times [0,\infty) \times [-d,0)\rightarrow \mathbb{R}
$$
which is measurable for the $\sigma$-algebra
$\textsf{ Prog}\otimes \mathcal{B}([-d,0))$, continuous in the third argument $s$, and locally bounded in time $t$
(e.g. continuous in time as well).
Then, as a consequence of the above definitions, for every admissible control $\phi(\cdot)$ the integral process $Y$:
$$
Y_t (\omega) = \int_0^t du \int_{-d}^0 N(\omega, u,s)\phi(u)(ds) \qquad t\geq 0
$$
is well defined, a.s. continuous and square integrable (it belongs to $\mathcal{L}^2_1$, see Definition \ref{not:basic}).
\subsection{ The agent's robust controls}
As opposed to the abstract formulation in Section 3,  we go back to  the notation $y$ for labor income as in Section 2.

The  SDDE for the labor income evolution under a Nature's control becomes
\begin{align}\label{EV-y-phi()}
\begin{split}
\left\{\begin{array}{ll}
dy(t) = & \left[ y(t) \mu_y+\int_{-d}^0 y(t+s)\phi(t)(ds)   \right] dt + y(t)\sigma_y^\top    dZ(t),\\[2mm]
y(0)= & x_0, \quad y(s) = x_1(s) \mbox{ for $s \in  [-d,0)$},
\end{array}\right. \end{split}
\end{align}
 in which $ y(t+s)(\omega)$  plays the role of $N(\omega, t,s)$.
  As already mentioned in Section 2.1 for the constant $\phi$ case, when $x_1 \in L^2([-d,0);\mathbb{R})$
  and $t\in [0,d)$,
  the  integral $\int_{-d}^0 y(t+s)\phi(t)(ds)$
  may not make sense.
 Exploiting the density of continuous functions however,  Proposition \ref{Prop-SDDE-app} shows that the Cauchy problem with
$$
x =(x_0,x_1) \in M_2
$$
not only makes sense but has a unique strong solution $y^x(\cdot;\phi(\cdot))$ with $P$-a.s. continuous paths.  Moreover, the feedback representation of $y^x(\cdot;\phi(\cdot))$ given in Proposition \ref{pr:ypositive} still holds:
\begin{equation}\label{EQ-y-phi()-feedback}
y^{x}(t;\phi(\cdot)) = E(t)(x_0 + I(t))
\end{equation}
with
$$
E(t)=e^{(\mu_y-\frac12 \sigma_y^\top \sigma_y)t + \sigma_y^\top Z(t)},
\quad
I(t) = I(t; \phi(\cdot))=\int_0^t E^{-1}(u) \left ( \int_{-d}^0 y^x(u+s;\phi(\cdot)) \phi(u)(ds) \right ) du
$$

When the  agent picks a   strategy $\pi\in \Pi^0$,
  $W^{w,x}(\cdot;\phi(\cdot),\pi)$ denotes  the corresponding solution of state equation for the wealth, and $\Gamma^{w,x}(\cdot;\phi(\cdot),\pi)$ the associated total wealth process.

 For a general  $\phi(\cdot)\in \mathcal{K}$ the Markovian representation of the total wealth  from a triplet $(c,B, \theta)$ as given in \eqref{eq:GammaG}  may not hold anymore. Therefore, for given initial conditions $(w,x)\in \mathcal{H}$, $\phi(\cdot)\in \mathcal{K}$ and
$\pi=(c, B,\theta)\in \Pi^0$, the total wealth $\Gamma^{w,x}(t;\phi(\cdot),\pi)$ must be calculated
as defined in \eqref{budget-constraint}. Given this,  the set of admissible strategies $\Pi(w,x;\phi(\cdot))$ is the generalization of \eqref{DEF_PI_FIRST_DEFINITION}. The \emph{robust} set of the agent's  controls is then defined as follows.

\begin{definition}
Given the initial data $(w,x)$, the agent's robust admissible controls are given by:
\begin{equation*}
\Pi^{rob}(w,x):= \bigcap_{\phi(\cdot) \in \cK}
 \Pi(w,x;\phi(\cdot))
\end{equation*}
\end{definition}
 Note that the set $\Pi^{rob}(w,x)$ is independent of the Natures's control $\phi(\cdot)$, as the set $\cK$ does not depend on the controls of the agent.
Similarly  to what happens with constant $\phi$,  the robust set of controls can be empty. Thus,  {we impose that } the order minimum measure $\nu$ verifies additional conditions.
 \begin{assumption}
 \label{ass:munew}
Assumption \ref{HYP_BETA-BETA_INFTY}-(i)
holds, with $\phi=\nu$, i.e. $\beta>\beta_\infty^{\nu}$, so $g_\infty^{\nu}>0$. Moreover
$h^{\nu}_\infty (s)\ge 0$ for all $s \in [-d,0]$, where $g^{\nu}_\infty, h^{\nu}_\infty$ are defined in \eqref{DEF_g_infty_h_infty} (with $\phi$ in place of $\nu$).
\end{assumption}
The next Lemma proves monotonicity of labor income with respect to  $\phi(\cdot)$.
\begin{lemma} \label{Labor-MON}
 Let $x\in M_2$ strictly positive and consider $\phi(\cdot), \psi(\cdot) \in \cK$, such that
$$
\phi(t) \leq \psi(t) \, \quad
 \  \ \P-\text{a.s. for all } t\geq 0
$$
Assume also that $y^x(t; \phi(\cdot)) > 0$ for all $t\ge 0$. Then, for all $t\ge 0$,
$$
y^x(t; \phi(\cdot)) \leq  y^x(t; \psi(\cdot)).
$$
Moreover, if we also ask that
\begin{equation}\label{eq:suppnotequal}
supp (\phi (t) )\subsetneq supp (\psi (t) )
 \  \ \P-\text{a.s. for all } t\geq 0
\end{equation}
then the inequality
$$
y^x(t; \phi(\cdot)) \leq  y^x(t; \psi(\cdot))
$$
is strict for $t>0$.
\end{lemma}
\begin{proof}
The proof is split in three steps.
\begin{enumerate}
  \item  First we prove the statements under the additional assumptions that: $x=(x_0, x_1)$ is continuous when seen as a function on $[-d,0]$; and, for all $T>0$,  there exists a constant $b_T>0$ such that
\begin{equation}\label{eq:phipsimag}
      \inf_{t\in [0,T]}\|\psi(t) -\phi(t) \|  >b_T, \quad \P -       a.s.
\end{equation}
The latter means that $\psi(\cdot)$ is \emph{ uniformly bigger }than $\phi(\cdot)$ on each compact time interval. Since we are assuming $\phi(t)\leq \psi(t)$ such uniform bound  implies the support condition \eqref{eq:suppnotequal}. \\
Now,  the datum $x$ is continuous on $[-d,0]$,   $y^x(t; \phi(\cdot))$ is continuous  on $[-d,+\infty)$. Such solution is strictly positive by assumption, so that  its minimum $m_T$ is strictly positive on each interval $[-d,T]$, $T>0$. Let  $z(t) = y^x(t; \psi(\cdot)) -y^x(t; \phi(\cdot)) $.  By the representation \eqref{EQ-y-phi()-feedback},
$$ z(t)=E(t)  \int_0^t du \, E^{-1}(u) \left(\int_{-d}^0 y^x(u+s; \psi(\cdot)) \psi(u)(ds)- \int_{-d}^0 y^x(u+s; \phi(\cdot)) \phi(u)(ds)\right)$$
Setting $z=0$ on $[-d,0]$ we can  rewrite the above as
 \begin{equation}\label{eq:zfeedback}
z(t)=  E(t) \int_0^t du \, E^{-1}(u) \left(\int_{-d}^0 z(u+s) \psi(u) (ds) + \int_{-d}^0 y^x(u+s; \phi(\cdot)) (\psi(u)-\phi(u))(ds)\right)
\end{equation}
which then verifies
\begin{equation}\label{eq:zfeedunif} z(t)\geq E(t) \int_0^t du \, E^{-1}(u) \left(\int_{-d}^0 z(u+s) \psi(u) (ds) + m_T  b_T \right)   \ \ 0\leq t\leq T
\end{equation}
Now fix $n \in \mathbb{N}$ and set $\tau_n = \inf \left \{ t \in [0,T] \mid z(t)< -\frac{ m_T  b_T }{ {c_T} (n+2)} \right \}$, in which $ {c_T}$ is the bound in total variation of the kernel $\psi(\cdot)$ on $[0,T]$. Since $z(0)=0$ and $z$ is continuous, it must be $\tau_n>0$ a.s. Continuity also implies
$z( \tau_n) = -\frac{ m_T  b_T }{ {c_T} (n+2)}$ where $\tau_n<\infty$. This however cannot happen because  by the inequality displayed above $z(\tau_n)>0$. Thus, $\tau_n =\infty$ a.s. for all $n$, which implies $z(t)\geq 0$ for all $t\leq T$.   Strict positivity of $z$   immediately follows from positivity and  the feedback inequality \eqref{eq:zfeedunif} for $z$.
Repeating the argument for all $T>0$ concludes the proof of this step.
  \item Suppose now that $x =(x_0,x_1)$ is a positive datum in $  M_2$, while we keep  the hypothesis \eqref{eq:phipsimag} that $\psi(\cdot)$ is uniformly bigger than $\phi(\cdot)$. Pick a sequence of continuous functions $h_n \in C([-d,0]; \mathbb{R}), h_n > 0$, $h_n(0)=x_0$ , with $h_n \rightarrow x_1 $ in $L^2([-d,0);\mathbb{R})$. Using the continuity of the solution $y^x$ with respect to the datum (see Proposition \ref{Prop-SDDE-app}), we can extract a subsequence, still denoted by $h_n$, such that
$$    y^{h_n}(t; \phi(\cdot))(\omega) \rightarrow  y^x(t; \phi(\cdot))(\omega)  \qquad    \text{ on } [0,T], \qquad \text{ $\P$- a.s.}   $$
By extracting once more if necessary, the above holds true also if  we substitute $\phi(\cdot)$ with $\psi(\cdot)$.
The first part of the statement then follows simply passing to the limit when $n \to +\infty$.
\\
By the first step above,
$$
y^{h_n}(t; \phi(\cdot)) < y^{h_n}(t; \psi(\cdot)).$$
Passing to the limit for $n\to +\infty$
we  obtain $y^x(t; \phi(\cdot))\leq y^x(t; \psi(\cdot))$.
Using the feedback inequality \eqref{eq:zfeedunif}
for the difference
$z = y^x(t; \psi(\cdot))- y^x(t; \phi(\cdot))$, we get strict positivity.

\item Now, consider a general datum $x$ and a general kernel $\psi(\cdot) \geq \phi(\cdot)$. Define a new kernel $\psi_n(\cdot)$ by adding to $\psi(\cdot)$  a constant (=not time dependent) term with flat density $\frac{1}{n+1}$ wrt the Lebesgue measure $ds$:  $$ \psi_n(t): = \psi(t) + \frac{1}{n+1}ds$$
    Then, each $\psi_n(\cdot)$ is uniformly bigger than $\phi(\cdot)$, with uniform lower bound $b_T(n)\geq \frac{1}{n+1}$. So, by step 2
     $$ y^x(t; \psi_n(\cdot))> y^x(t; \phi(\cdot)) $$

       From  the proof of  Proposition \ref{Prop-SDDE-app}, the integral map $F$ with general delay kernel is a contraction in $L^p(\Omega; C(([0,T]; \mathbb{R}))$ for each $p>2$. The  contraction constant $0<C_T(\alpha)<1$ depends a sufficiently large coefficient $\alpha>0$, which in turn depends  on the delay kernel only through its bound on total variation  $ {c_T}$.  Here, if we consider $\psi_n(\cdot)$ as delay kernel, the bound on its total variation is $(c_n^*)_T : = { c_T} + \frac{1}{n+1} \leq  {c_T} +1 $ . This implies that we can find a common contraction coefficient $C_T(\alpha)$, for all $n$.   Call $F^{\psi_n}$ the integral map with delay $\psi_n(\cdot)$, and similarly for $F^{\psi}(\cdot)$.
      Then we have, in the $\alpha$ norm on $L^p(\Omega; C(([0,T]; \mathbb{R}))$ (see \eqref{eq:defnormalfa}):
     \begin{align*}
      \| y^x(t; \psi_n(\cdot))-  y^x(t; \psi(\cdot))\|_{\alpha} & = \| F^{\psi_n}(y^x(t; \psi_n(\cdot)))-F^{\psi}(y^x(t; \psi(\cdot)))\|_{\alpha} \\
      & \leq \|F^{\psi_n}( y^x(t; \psi_n(\cdot)) )-F^{\psi_n}( y^x(t; \psi(\cdot)) )\|_{\alpha} \\ & +
     \|F^{\psi_n}( y^x(t; \psi(\cdot)) )-F^{\psi}(y^x(t; \psi(\cdot)))\|_{\alpha}\\
     & \leq C_T(\alpha) \| y^x(t; \psi(\cdot)) - y^x(t; \psi_n(\cdot)) \|_\alpha+ \frac{1}{n}\| y^x(t; \psi(\cdot)) ) \|_\alpha T  \end{align*}
   This can be deduced from the proof of Proposition \ref{Prop-SDDE-app}, specialized to the present case. So,
 $$ \| y^x(t; \psi_n(\cdot))-  y^x(t; \psi(\cdot))\|_{\alpha} \leq \frac{1}{1- C_T(\alpha)} \frac{1}{n}\|y^x(t; \psi(\cdot))\|_{\alpha}  T  $$
 Therefore  $y^x(t; \psi_n(\cdot)) \rightarrow y^x(t; \psi(\cdot))$ in $L^p$,   modulo an extraction  we can pass to the a.s. limit for $n\to + \infty$ to conclude, once again, that
 $$ y^x(t; \psi(\cdot)) \geq y^x(t; \phi(\cdot))$$
 When the support condition \eqref{eq:phipsimag} is verified, by strict positivity of $y^x(t; \phi(\cdot))$   and the feedback formula \eqref{eq:zfeedback} we get  that the above inequality is strict.
\end{enumerate}

\end{proof}

\begin{proposition}
\label{Pirob}
Let $w\in \R$. Let also $x\in M_2$ be strictly positive and such that $y^x(t;\nu)>0$, $P$-a.s., for all $t \ge 0$. Then, under assumption \ref{asK},
\begin{equation}\label{eq:propRob1}
\Pi^{rob}(w,x)=\Pi(w,x;\nu)
\end{equation}
As a consequence,
\begin{equation}\label{eq:propRob2}
\Pi^{rob}(w,x)\ne \emptyset
\quad \Longleftrightarrow \quad
(w,x) \in \cal H^{\nu}_+,
\end{equation}
where $\cal H^{\nu}_+$ is the admissible set of initial data for $\nu$:
\begin{equation*}
\mathcal{H}^{\nu}_+ =\{ (w,x) \mid
\Gamma^{w,x}(0;\nu,\pi)= G^{\nu}(w,x) \geq 0\}
\end{equation*}
as defined at the end of Section 3.1.
\end{proposition}
\begin{proof}
The proof is based on a monotonicity argument.
 We first prove \eqref{eq:propRob1}.
Since $\nu$ is the order minimum, any fixed $\phi(\cdot)\in \mathcal{K}$ verifies
$$ \nu \leq \phi(t)(\omega) \qquad \forall{(\omega, t)}$$
 By Lemma \ref{Labor-MON}  the labor income $y^x(\cdot ;\phi(\cdot))$ is monotone increasing with respect to $\phi(\cdot)$, for any $\phi(\cdot) \in \cal K$. In turn,  the wealth $W$ is monotone increasing in $y$. This  easily implies
 (using \eqref{budget-constraint})  that the total wealth $\Gamma^{w,x}(t;\phi(\cdot),\pi)$ is also monotone increasing in $\phi(\cdot)$. From the very  definition, a strategy $\pi$ belongs to $\Pi^{rob}(w,x)$ if it verifies the total wealth positivity constraint for all $\phi(\cdot)$. This holds  if and only if the strategy satisfies the admissibility constraint for the kernel corresponding to the smallest total wealth:
$$
\Gamma^{w,x}(t;\nu,\pi)\ge 0, \qquad \forall t\ge 0
$$
which shows $\Pi^{rob}(w,x)=\Pi(w,x;\nu)$.
\\
For the  statement \eqref{eq:propRob2},
a necessary condition for
$\Pi^{rob}(w,x)\neq \emptyset$
is  that the datum $(w,x) \in \cal H^{\nu}_+$.
We show that it is also sufficient.
In fact, if  $(w,x) \in \cal H^{\nu}_+$ then  Theorem \ref{TEO_MAIN_INF_RET}  applied with $\phi=\nu$ provides the
optimal triplet $\pi^* \in \Pi(w,x;\nu)$ under the move $\nu$ for Nature:\footnote{For the general correlation case, just replace the $\theta^*$ given here with its general form given in Remark \ref{VICEIRA}.}
$$
c^*(t) = f_{\infty}^{-1} (\Gamma^{w,x})^{*}(t;\nu), \ \qquad B^*(t) = k^{-b}(\Gamma^{w,x})^{*}(t;\nu)
$$
and
$$
\theta^*(t) = (\sigma\sigma^\top)^{-1}(\mu-r \mathbf 1)
\frac{(\Gamma^{w,x})^{*}(t;\nu)}{\gamma} -g^{\nu}_{\infty} y^x(t;\nu)\left(\sigma^{\top}\right)^{-1} \sigma_y
$$
 where $(\Gamma^{w,x})^{*}(t;\nu)$ is defined in \eqref{Gamma-infty2} and is the solution of \eqref{DYN_GAMMA*_PB1}
with $\nu$ in place of $\phi$ in the initial condition.  By the first statement however such $\pi^* \in \Pi^{rob}(w,x)$, which concludes the proof.

\end{proof}


\subsection{ Solution of the robust problem}
As observed above, the sets of agent's and Nature's controls are mutually independent and then we can formulate the robust problem as a static game.
Recall that the objective function is
\begin{align}\label{eq:Jgame}
J(\pi;\phi(\cdot))=
\E \left[\int_{0}^{+\infty} e^{-(\rho+ \delta) t }
\left( \frac{c(t)^{1-\gamma}}{1-\gamma}
+ \delta \frac{\big(k B(t)\big)^{1-\gamma}}{1-\gamma}\right) dt
\right]
\end{align}
with $\pi=(c,B, \theta) \in \Pi^{rob}(w,x)=\Pi(w,x;\nu)$  and $\phi(\cdot) \in \mathcal{K}$. Note also that $J$ does not explicitly depend on $\phi(\cdot)$.
The static lower value of the game is
$$
L =L(w,x) := \sup_{ \pi\in \Pi^{rob}(w,x)}\;
\inf_{\phi(\cdot) \in \cK} J(\pi; \phi(\cdot)),
$$
which is clearly less or equal than the static upper value
$$
U= U(w,x) : =  \inf_{\phi(\cdot) \in \cK}
\;\sup_{\pi  \in \Pi^{rob}(w,x) }  J(\pi; \phi(\cdot))
$$
The game has a value when $L=U$. The following Proposition shows that this is indeed the case, as one intuitively may have guessed from the minimality of $\nu$.
\begin{proposition}
 Let $w\in \R$. Let also $x\in M_2$ be strictly positive and such that $y^x(t;\nu)>0$, $P$-a.s., for all $t \ge 0$. Then, under assumption \ref{asK},
the static Agent vs Nature game has a value  $V=L=U$,  and there exists a saddle point, solution of the game, given by
$$ (c^*, B^*,\theta^*; \nu) $$
in which the optimal agent's triplet is the one given in \eqref{OPTIMAL_STRATEGIES_RET_INF} with $\phi =\nu$. As a consequence, the agent becomes observationally equivalent to one who has worst case beliefs on the influence of past wages on the present.
\end{proposition}
\begin{proof}
  Consider the optimal  triplet $(c^*, B^*, \theta^*)$  as from Proposition \ref{Pirob}.  This strategy is in $\Pi^{rob}(w,x)$ and the following chain holds:
$$
\max_{\pi \in \Pi^{rob}(w,x)} J(\pi; \nu) = J(c^*,B^*, \theta^*; \nu) =  \min_{\phi(\cdot) \in \cal{K}} J(c^*,B^*, \theta^*; \phi(\cdot))\leq L,
$$
where  the second equality  follows from the monotonicity with respect to  $\phi(\cdot)$ of the total wealth.  Thus,
$$ \max_{\pi \in \Pi^{rob}(w,x)} J(\pi; \nu) \leq  L \leq U \leq   \inf_{ {\phi(\cdot) \in  \mathcal{K}}} \sup_{\pi \in \Pi^{rob}(w,x)}  J(\pi;  {\phi(\cdot)}) \leq
  \max_{\pi \in \Pi^{rob}(w,x)} J(\pi; \nu) $$
  which concludes the proof.
\end{proof}

\appendix

\section*{Appendix}

\section{Proof of Proposition \ref{neg}}
\label{app:lemmapos}

We need first an auxiliary measure theoretic lemma, of which we provide a proof for the reader's convenience.  The Hahn-Jordan decomposition of a measure into positive and negative part can be found in \cite{Hahn}.
\begin{lemma}
Let $\phi \in \mathcal{M} $ and let $\phi^+,\phi^-$ be its Hahn-Jordan decomposition. Assume $\phi^-\neq 0$,
 and call $m$ the mass of $\phi^-$, namely $ m = \phi^-([-d,0])>0$.  Then,
  there exists a nonnegative continuous function $x^*_1$ on $[-d, 0]$ such that
  $$ \int_{-d}^0 x^*_1(s) \phi(ds) < - \frac{m}{2}$$
\end{lemma}
\begin{proof}
 If $\phi^+$ is null, just take $x_1^*=1$. Otherwise, let $C=\mathrm{supp}( \phi^{-})$ be the support of $\phi^-$.  This is the smallest closed (proper) subset $D$  of $[-d,0]$ such that $\phi^-(D) =m$.
 Since $C$ is closed, there exists  a sequence of open sets, $(A_n)_n$  which decreases  to $C$. By standard topological separation properties, for every $n$ one can find a continuous function
 $$f_n: [-d,0] \rightarrow [0,1] $$
 with  $f_n^{-1}(1)=C, f_n^{-1}(0)= A_n^c$.  Since $A_n \downarrow C$,
 $$f_n(s) \rightarrow I_C(s) \text{ pointwisely } $$
  By bounded convergence:
 $$ \int_{-d}^0 f_n(s) \phi(ds)\rightarrow \int_{-d}^0 I_C(s) \phi(ds) =  - m $$
 and therefore there exists an $n^*$ s.t. $  \int_{-d}^0 f_n(s) \phi(ds) < -\frac{m}{2}$ for $n\geq n^*$. Now, take $x_1^* = f_{n^*}$.
\end{proof}

\paragraph{Proof of Proposition 2.9}
Recall the representation \eqref{VARIATION_CONSTANTS_FORMULA} for $y$:
 $$ y(t)= y(t)^{(x_0,x_1)}=E(t)\left (x_0+ \int_0^t     E^{-1}(u) \left(\int_{-d}^0y(u+s) \phi(ds)\right) du \right) =  E(t)\big(x_0 + I(t)\big)$$
 in which $E(t) = e^{(\mu_y - \frac{1}{2} \sigma_y^\top \sigma_y) t + \sigma_y^\top Z(t)} $. As already noted before the statement of Proposition \ref{neg}, this representation  gives directly the implication: when $\phi\geq 0$, for every initial data $x_0>0, x_1\geq 0$, the labor income $y$ has positive paths $P$-a.s. The converse implication is proved hereafter by contradiction. \\
 \indent Assume $\phi \in \radon^0$ is not positive. Equivalently, it  has Hahn-Jordan  decomposition $\phi =\phi^+-\phi^-$  with $\phi^-([-d,0]) = m >0$. Then,  we look for a suitable $x=(x_0,x_1)$ with $x_0>0, x_1\geq 0$ so that the trajectory of $y$ crosses the $t$-axis with positive probability. The idea is to pick $x_0$ sufficiently small with respect to  $x_1$.  Let  $x_1 =x_1^*$ be the continuous function from the above Lemma, and fix $x_0 = c\frac{ m}{8}$, in which $c >0$  is a parameter which will be chosen later. \\
 \indent The average of the past is  negative,  $ \int_{-d}^0 x_1(s) \phi(ds) < -\frac{m}{2} $. Define now two stopping times:
 $$ \tau_c :=\inf \left \{ t \,  \middle \vert \, \int_{-d}^0 y(t+s) \phi(ds) > - \frac{m}{4} \right \}$$
 $$\rho: = \inf \left \{ t \mid E^{-1}(t) < \frac{1}{2} \right \}   $$
 Clearly, $\tau_c, \rho$ are both a.s.  strictly positive and by linearity of the SDDE,  $\tau_c$ depends only on $c$ as $m$ simplifies. Define $\varrho_c:=\tau_c \wedge \rho \wedge n_0$, where $n_0 \in \mathbb N_+$ ensures boundedness of $\varrho$. The pathwise  relation holds a.s.
 $$  I(\varrho_c \wedge t) = \int_{0}^{\varrho_c \wedge t} du\, E^{-1}(u) \int_{-d}^0 y(u+s)\phi(ds) < - \frac{m}{8} \varrho_c \wedge t, $$
 and in particular:
 $$  y(\varrho_c ) =  E(\varrho_c )(x_0 +I(\varrho_c)) < E(\varrho_c) \left (x_0 - \frac{m}{8} \varrho_c  \right ) = \frac{m}{8}  E(\varrho_c)  (c-  \varrho_c) $$
 To conclude, we show that for some $c>0$ we have $P(y(\varrho_c )<0) = P( c- \varrho_c<0)>0$.
 Note that   $\tau_c$ on $c$ is well defined for every $c \in \mathbb{R}$ and monotone non decreasing in $c$. Consequently,   $ \varrho_c $ has the same characteristics.  For $c=0$, a quick look at the definition shows that
  $$ P( 0 < \varrho_0)=1 $$
  Now, fix  $n \in \mathbb{N}_+$, take $c=\frac{1}{n}$ and  consider  $\varrho_{\frac{1}{n}}$. Then, $\varrho_{\frac{1}{n}} \uparrow \varrho_0 $ and thus
  $$ \lim_n P( \varrho_{\frac{1}{n}}- \frac{1}{n} >0) = P(\varrho_0 >0)= 1 $$
  This concludes the proof, since for every $1> p>0$  we can find $n$ large enough such that the constant $c= \frac{1}{n}$ verifies: $ P( \varrho_{\frac{1}{n}}- \frac{1}{n} >0) > p$.

 \section{Well posedness of the SDDE for labor income}
 \label{app:wellposed}
 The focus here is the well posedness of \eqref{EV-y-phi()} with   non autonomous, stochastic kernel $\phi(\cdot)$.
The existing literature on the well posedness of this type of equation is quite rich. The results vary according to the deterministic or stochastic setup and  the hypotheses on $x$ and $\phi(\cdot)$.  In the context of deterministic delay equations, the results are on
\begin{itemize}
\item  the autonomous case, i.e. constant  $\phi\in \mathcal{M}$, see \cite[Section 3.2]{BENSOUSSAN_DAPRATO_DELFOUR_MITTER};
\item  the non autonomous case, i.e.  when $\phi$ is a time dependent Radon measure, and the datum  $x\in \R\times L^2([-d,0);\R)$, is treated in \cite{HaddSICON06}.
\end{itemize}

In the stochastic case when $\phi$ is constant we are aware of the following results:
\begin{itemize}
\item when $x\in \R\times L^2([-d,0) ;\R)$ \ and $\phi$ is absolutely continuous  with respect to the Lebesgue measure : $d\phi = \varphi(s) ds$, with
    $\varphi \in L^2([-d,0);\R)$,
    the existence and uniqueness result follows from \cite[Theorem I.2 and Remark 3-(iv, p.18]{MOHAMMED_BOOK_96}.
\item when the initial datum $x$ is Borel measurable and bounded on $[-d,0]$, and $\phi(\cdot)$ is as in Definition \ref{df:adm}, the well-posedness of \eqref{EV-y-phi()} is proved in
    \cite[Theorem 3.6]{Rosestolato17} in a general nonlinear framework.
\end{itemize}
 We rewrite here below \eqref{EV-y-phi()} for the reader's convenience:
\begin{equation}
\label{SDDE_app}
\begin{cases}
{\rm d}y(t)=\left[y(t)\mu_y+ \int_{-d}^0y(t+s)\phi(t)(ds)\right]\,{\rm d}t + y(t)\sigma_y^\top\,{\rm d}Z(t)
\\
y(0)=x_0, \quad y(s)=x_1(s) \ \text{for} \ s \in[-d,0).
\end{cases}
\end{equation}
 Here  $x=(x_0, x_1) \in M_2= \mathbb{R} \times L^2([-d,0); \mathbb{R})$ and $\phi(\cdot) \in \mathcal{K}$ (see Definition \ref{df:controlNature}).
We do not restrict to positive data, as the result holds in general.  \\
 \indent Let us   introduce a handy notation for the past path at $t$ of a (deterministic) function $f:[-d,T] \rightarrow \mathbb{R}$, for $0\leq t\leq T$,
 \begin{equation*}
  f_t(s) := f(t+s) \, \,   \, \text{ for } -d \leq s \leq 0.
 \end{equation*}
    The past path of $y$ at $t$ for the realization $\omega$ is thus $y_t(s,\omega): =y(t+s,\omega) \, \, s \in [-d,0]$. The delay term  in the drift  reads as follows.  The pathwise integral
    \begin{equation}
    \label{path}
    \int_{-d}^0y(t+s,\omega)\phi(t,\omega)(ds) =\int_{-d}^0y_t(s,\omega)\phi(t,\omega)(ds)
    \end{equation}
     of the realized  past path $y_t$ is made with respect to the realized measure $\phi(t,\omega)(ds)$, revealed at time $t$.

The delay part in \eqref{SDDE_app} can be expressed in terms of  (an extension of) a progressively measurable stochastic process whose values are linear non-autonomous  operators of kernel-type:
\begin{equation*}
L(t,\omega):C([-d,0];\mathbb{R}) \rightarrow \mathbb{R},
\end{equation*}
\begin{equation}
\label{L}
L(t, \omega)f=\int_{-d}^0 f(s)\,\phi(t,\omega) ({\rm d}s), \qquad \forall f \in C([-d,0]; \mathbb{R}),
\end{equation}
defined for every $t \ge 0$ and \gre{$\omega \in \Omega$.} 
\footnote{ Note that here, by Definition \ref{df:adm},
the integral is defined for all $\omega \in \Omega$, not just $ \mathbb{P}$-a.s.}

Notice that the operator $L$ given in \eqref{L} is defined  on the space   $C([-d,0]; \mathbb{R})$. When the initial datum $x_1$ is not continuous, but only square integrable with respect to the Lebesgue measure, problems may arise.
In fact, consider an initial datum $(x_0, x_1) \in M_2$  and proceed heuristically by  assuming that the solution to \eqref{SDDE_app} exists. For $ 0\leq t<d $, the past path is denoted by
\begin{equation}
\label{path_window}
 y_t(s) =
\begin{cases}
y(t+s), &  -t \le s  \le 0,
\\
 y(t+s) = x_1(s)  , &   -d \le s <-t,
\end{cases}
\end{equation}
which in general is not a continuous function, but only square integrable. Thus the  operator $L$ cannot be applied to $y_t$ as the integral in \eqref{path} (and in \eqref{SDDE_app}) may be not well defined.
In other words, we cannot expect to give a pointwise meaning to the function $t \mapsto L(t,\omega)y_t$, when the initial datum of problem \eqref{SDDE_app} belongs to $M_2$.
\indent Lemma \ref{newlemma.} below shows that the delay operator admits a continuous extension to $L^2([-d,T];\mathbb{R})$.  This Lemma  is a   generalization of
\cite[Theorem 3.3-(iii), p.249]{BENSOUSSAN_DAPRATO_DELFOUR_MITTER}.
 In the proof, we need the following notations and properties:
\begin{itemize}
  \item  the absolute value $|\phi|$  of any $\phi \in \mathcal{M}$ is the measure  given by the sum $\phi^+ +\phi^-$;
  \item  when a family  of measures $(\psi_r)_{r\in I}$ is bounded in total variation, its supremum $\widetilde{\psi}$:
$$ \widetilde{\psi}(A):= \sup_{r\in I} \, \psi_r(A), \ \ A \in \mathcal{B}([-d,0))$$
is a Radon measure as well - even when $I$ is uncountable. In fact, $\mathcal{M}$ is a Banach lattice and the well posedness of the supremum follows from the  countable sup property of Banach lattices (see e.g. \cite{AB}[Theorem 8.22]).
\end{itemize}

\begin{lemma}
\label{newlemma.}
 Let $\phi(\cdot)$ be admissible. For $t \ge 0$
and $\omega \in \Omega$, let $L(t,\omega)$  be the linear and continuous map from $C([-d,0];\mathbb{R})$ into $\R$  defined in \eqref{L}. Fix $T>0$ and $\omega \in \Omega$. Define the operator
$$
\mathcal{N}_T(\omega):C([-d,T]; \mathbb{R})\rightarrow
L^2([0,T];\mathbb{R})
$$
as follows. For $z \in C([-d,T]; \mathbb{R})$
\begin{equation}\label{eq:Delfour}
(\mathcal{N}_T(\omega)z)(t)
:=L(t, \omega)z_t, \quad  \,  0\leq t\leq T
\end{equation}
Then, for every $\omega \in \Omega$ the following  hold.
\begin{itemize}
\item [i)]
The map
$ \mathcal{N}_T(\omega): C([-d,T];\mathbb{R})\rightarrow L^2([0,T];\mathbb{R})$
is well defined and satisfies the $L^2$-inequality:
\begin{equation*}
\|\mathcal{N}_T(\omega)z\|_{L^2([0,T];\mathbb{R})} \le c_T\|z\|_{L^2([-d,T];\mathbb{R})}, \qquad \forall y\in C([-d,T];\mathbb{R})
\end{equation*}
in which $ c_T$ is the constant in Definition \ref{df:adm}.
\item [ii)]
By item i), the operator $\mathcal{N}_T(\omega)$ admits an $L^2$-norm continuous, linear extension (denoted in the same way) to $L^2([-d,T];\mathbb{R})$.
\end{itemize}
\end{lemma}

\begin{proof}
\begin{itemize}
\item [i)]
\begin{align*}
\|\mathcal{N}_T(\omega)(z)\|_{L^2([0,T];\mathbb{R})}
= & \, \|L(\cdot,\omega)z\|_{L^2([0,T];\mathbb{R})}\\
= & \sup_{f\in L^2([0,T];\mathbb{R}),\,\|f\|_{L^2}=1}\; \int_0^T\ {\rm d} r \,  f(r) \,   \int_{-d}^0 z_r(s) \phi(r,\omega) ({\rm d}s) \\
\leq & \sup_{f\in L^2([0,T];\mathbb{R}),\,\|f\|_{L^2}=1}\;
\int_0^T\ {\rm d} r \,  |f(r)| \, \left |  \int_{-d}^0 z_r(s) \phi(r,\omega) ({\rm d}s)\right |
\end{align*}
where the second equality follows by definition of (dual) norm in $L^2([0,T];\mathbb{R})$, and the inequality by continuity of the integral wrt ${\rm d}r$. Again by continuity of the integral (wrt $\phi(r,\omega)({\rm d}s)$) and by monotonicity,   the following holds for all $f \in L^2([0,T];\mathbb{R})$:
\begin{align*}
\int_0^T\ {\rm d} r \,  |f(r)| \left |  \int_{-d}^0 z_r(s) \phi(r,\omega) ({\rm d}s)\right | & \leq
\int_0^T {\rm d}r |f(r)| \int_{-d}^0 |z(r+s)||\phi(r,\omega)| ({\rm d}s)\,
\\
& \leq  \int_0^T {\rm d}r |f(r)| \int_{-d}^0 |z(r+s)| \sup_{r\in [0,T]} |\phi(r,\omega)| ({\rm d}s)
\end{align*}
Now, the sup-measure $\widetilde{\phi}(\omega):= \sup_{r\in [0,T]} |\phi(r,\omega)|$ is autonomous - does not depend on $r$, so  we can apply the Fubini Tonelli Theorem and develop further the inequality:
\begin{align*}
\int_0^T\ {\rm d} r \,  |f(r)| \left |  \int_{-d}^0 z_r(s) \phi(r,\omega) ({\rm d}s) \right | & \leq  \int_{-d}^0 \widetilde{ \phi}(\omega) ({\rm d}s)  \int_{0}^T |f(r)|\,|z(r+s)|{\rm d} r
\end{align*}
Passing to the supremum over $f$, we have
\begin{align*}
 \|\mathcal{N}_T(\omega) z \|_{L^2([0,T];\mathbb{R})} & \leq   \sup_{f\in L^2([0,T];\mathbb{R}),\,\|f\|_{L^2}=1}
   \int_{-d}^0 \widetilde{ \phi}(\omega) ({\rm d}s)  \int_{0}^T |f(r)|\,|z(r+s)|dr \\
   & \leq \int_{-d}^0 \widetilde{ \phi}(\omega) ({\rm d}s)  \sup_{f\in L^2([0,T];\mathbb{R}),\,\|f\|_{L^2}=1} \int_{0}^T |f(r)|\,|z(r+s)|dr\\
   & \leq \int_{-d}^0 \widetilde{ \phi}(\omega) ({\rm d}s)    \| \,|z|\,  I_{[s, T+s]}\|_{L^2([-d,T]; \mathbb{R})} \\
   & \leq
  \int_{-d}^0 \widetilde{\phi}(\omega) (ds)\, \|\,|z|\,\|_{L^2([-d,T]; \mathbb{R})}\\
  & \leq c_T    \|z\|_{L^2([-d,T]; \mathbb{R})}
\end{align*}
In the last passage  we used that $\phi(\cdot)$ is bounded in total variation by $c_T$  over $[0,T]$.
\item [ii)]
In view of assertion (i), the existence of the bounded linear extension of $\mathcal{N}_T(\omega)$ to $L^2([-d,T];\mathbb{R})$ immediately follows by the inequality \eqref{eq:Delfour} and by the density of $C([-d,T];\mathbb{R})$ in $L^2([-d,T];\mathbb{R})$.
\end{itemize}
\end{proof}

 Given any fixed $\omega \in \Omega$ to understand the explicit action of $\mathcal{N}_T(\omega)$ on a general $z \in L^2([-d,T];\mathbb{R})$ one has to resort to its definition. The action of $\mathcal{N}_T(\omega)$ on continuous functions is clear. Take a sequence of continuous functions on $[-d,T]$, $z_n$, which tend in $L^2$ to $z$. Since the delay map on general $z$ is a continuous extension, then
$$
\mathcal{N}_T(\omega)z_n \longrightarrow
\mathcal{N}_T(\omega)z    \ \ \text{ in } L^2([0,T];\mathbb{R}).
$$
Below we see a couple of examples on this.

\begin{enumerate}
\item Consider the deterministic kernel
   $$\phi(t,\omega) =  \frac{T-t}{T}\delta_{-d}$$
   Such $\phi$ models a memory which depends only on what happened $d$ instants ago, and fades in size as $t$ goes to $T$. Clearly $c_T=1$.  Now, for any $\omega \in \Omega$ we have
   $$
     (\mathcal{N}_T(\omega)z) (t) =\frac{T-t}{T}  z(t-d),\qquad \text{ for all } t\in[0,T]
   $$
   if $z$ is continuous. It is then easy to check that the extension of $\mathcal{N}_T(\omega)$ to $L^2([-d,T];\mathbb{R})$ is
  $$ (\mathcal{N}_T(\omega)z) (t) =\frac{T-t}{T}  z(t-d), \qquad  \text{   for a.e. } \, t\in[0,T] $$
  A similar result holds when the atomic measure is a linear combinations of Dirac deltas on $[-d,0]$.
\item Let $Z$ be the Brownian motion driving the stock market. Let $\phi$ be the kernel
       $$ \phi(t,\omega)(ds)  = \frac{1}{1+Z^2(t)(\omega)} \rm{d}s $$
       Namely, $\phi(t,\omega)$ is absolutely continuous wrt $ds$, with flat density $\frac{1}{1+Z^2(t)(\omega)}  $. This family of measures also has $c_T=1$ as an upper bound for the total variation norm. Here, for any $\omega \in \Omega$ and any
       $z \in L^2([-d,T];\mathbb{R})$ we have
   $$
   (\mathcal{N}_T(\omega)z) (t) =  \frac{1}{1+Z^2(t)(\omega)} \int_{-d}^0 z(t+s) \rm{d}s,
   \qquad \text{ for a.e.} \, t\in[0,T]
   $$
\end{enumerate}

\begin{proposition}\label{Prop-SDDE-app}
Consider an admissible kernel $\phi(\cdot)$.
For any given initial datum $x=(x_0,x_1)\in M_2$,
 the SDDE
\begin{equation}
\label{SDDE_appendix}
\begin{cases}
{\rm d}y(t)=\left[y(t)\mu_y+ \int_{-d}^0y(t+s)\phi(t)({\rm d}s)\right]\,{\rm d}t + y(t)\sigma_y^\top\,{\rm d}Z(t)
\\
y(0)=x_0, \quad y(s)=x_1(s) \ \text{for} \ s \in[-d,0),
\end{cases}
\end{equation}
admits a unique strong (in the probabilistic sense) solution $y$ in the space $L^2(\Omega;C([0,T];\mathbb{R}))$,  which depends continuously on $x$.   Moreover, such solution belongs to $L^p(\Omega;C([0,T];\mathbb{R}))$ for all $p\ge 2$.
\end{proposition}

\begin{proof}

Fix any $x=(x_0,x_1)\in M_2$. Let $T>0$ and let $S_T$ be the space
\begin{equation*}
S_T:=\{f \in C([0,T];\mathbb{R}): f(0)=x_0\},
\end{equation*}
endowed with the norm
$$ \|f\|_{S_T,{\alpha}} = \sup_{0\leq t \leq T}
 \, e^{-\alpha t}|f(t)|$$
for a real constant $\alpha>$ to be chosen later. Such norm is equivalent to the standard sup norm.
On $L^p(\Omega;S_T)$ ($p\ge 2$), consider the  norm
\begin{equation}\label{eq:defnormalfa}
\|y\|_{\alpha}:=(\mathbb{E}[\,  \| y(\cdot)\|^p_{S_T,\alpha} ])^{\frac{1}{p}} =\left ( \mathbb{E}\left [  \sup_{0\leq t \leq T}
(e^{-\alpha t}|y(t)|)^p   \right ] \right )^{\frac{1}{p}}
\end{equation}
We denote by $p':=p/(p-1)$ the conjugate exponent to $p$. For $y$ in $L^p(\Omega;S_T)$, we define a new process $F(y)$ as follows:
\begin{equation}
\label{eq:integral-map}
F(y)(t):=x_0+\mu_y\int_0^ty(r)\, {\rm d}r+\int_0^t (\mathcal{N}_T(\cdot)\bar y^{x_1})(r)\, {\rm d}r + \int_0^t y(r)\sigma_y ^\top\, {\rm d}Z(r), \qquad \quad 0 \le t\le T.
\end{equation}
In the above expression $\mathcal{N}_T(\omega)$ is the continuous linear operator introduced in Lemma \ref{newlemma.}(ii) and
$\bar y^{x_1} \in L^p(\Omega;L^2([-d,T];\mathbb{R}))$
is defined by pasting $y$ and $x_1$:
$$ \bar{y}^{x_1}(t) = \begin{cases}
                        x_1(t), &  \text{if }    -d\leq t< 0; \\
                         y(t),  &  \text{if }  \  \  0\leq t\leq T.
                      \end{cases}
$$
Now we take $p >2$ and we show that the following contraction inequality holds: there exists $k\in (0,1)$ such that, for all $y,z\in L^p(\Omega;S_T)$:
$$
\| F(y) -F(z)\|_{\alpha} \leq k \|y-z\|_{\alpha}.
$$
Since $F(0)\equiv x_0$, this shows that  $F$ maps $L^p(\Omega;S_T)$ into itself for all $p>2$ and that it is a contraction on this space when $p>2$. By the Banach fixed point Theorem, this implies that there exists a unique $y\in L^p(\Omega;S_T)$ such that $F(y)=y$, or, equivalently, $y$ is the unique solution,
in $L^p(\Omega;S_T)$, to the SDDE
\begin{equation*}
y(t)=x_0+\mu_y\int_0^ty(r)\, {\rm d}r+\int_0^t (\mathcal{N}_T(\cdot)\bar y^{x_1})(r) {\rm d}r + \int_0^t y(r)\sigma_y ^\top\, {\rm d}Z(r), \quad \ 0 \le t \le T, \quad \text{$P$-a.s.}
\end{equation*}
Given $y,z \in L^p(\Omega;S_T)$, by the definition of $F$ and of the norm
in \eqref{eq:defnormalfa} we get
\begin{align} \label{eq:alfanewestimate}
 \|F(z)-F(y)\|_{\alpha}^p
 \le
&3^{p-1}
\mathbb{E}
  \left[\sup_{t \in [0,T]}
  e^{-p\alpha t}  |\mu_y|^p\,\left|\int_0^{t}(z(r)-y(r)){\rm d}r\right|^p \right]
 ]   \\
 & +
3^{p-1} \mathbb{E}
  \left[\sup_{t \in [0,T]} e^{-p\alpha t} \left|\int_0^{t} (\mathcal{N}_T(\cdot)(\bar z^{x_1}-\bar y^{x_1}))(r)\, {\rm d}r \right|^p \right]
\notag
\\
&
+3^{p-1}\mathbb{E}\left[\sup_{t \in [0,T]} e^{-p\alpha t}\left|\int_0^{t} (z(r)-y(r))\sigma_y ^\top\, {\rm d}Z(r)\right|^p\right]
\end{align}
We now estimate the first term of the right hand side of \eqref{eq:alfanewestimate}:
\begin{align}
\label{eq:alfanewestimate1}
&\mathbb{E}
  \left[\sup_{t \in [0,T]} e^{-p\alpha t}
  \left|\int_0^{t}(z(r)-y(r)){\rm d}r\right|^p
\right]
\leq
\mathbb{E}\left [ \sup_{t \in [0,T]}
  \left|\int_0^{t}e^{-\alpha (t-r)}e^{-\alpha r}(z(r)-y(r)){\rm d}r ^p \right| \right ]
\notag
\\
&\qquad\stackrel{\text{ H\"{o}lder}}{\le}
\mathbb{E}
\left[\sup_{t \in [0,T]}\left(\frac{1-e^{-p'\alpha t}}{p'\alpha}\right)^{p/p'}
  \int_0^{t}e^{-p\alpha r}|z(r)-y(r)|^p{\rm d}r
\right]
\\
\notag
&\qquad\le \mathbb{E}
  \left[\left(\frac{1}{p'\alpha}\right)^{p/p'}
  \int_0^{T}e^{-p\alpha r}|z(r)-y(r)|^p{\rm d}r
\right]
\le T\left(\frac{1}{p'\alpha}\right)^{p/p'}  \|z-y\|_{\alpha}^p.
\end{align}
For the estimate of the second term of \eqref{eq:alfanewestimate},
 note first that  the definition \eqref{path_window} implies
 \begin{equation}\label{null}\bar z^{x_1}_r(s)-\bar y^{x_1}_r(s)=0  \text{ if }  r+s<0 \text{ and
 } z(r+s)-y(r+s) \, \text{ otherwise},
 \end{equation}
so that $\| \bar z^{x_1}_{\cdot}(s) -\bar y^{x_1}_{\cdot}(s)\|_{S_T, \alpha}^p $ is finite. Then,
\begin{align}
\label{eq:alfanewestimate2}
&\mathbb{E}
  \left[\sup_{t \in [0,T]} e^{-p\alpha t}
\left|\int_0^t (\mathcal N_T(\cdot)(\bar z^{x_1}-\bar y^{x_1})(r)\, {\rm d}r \right|^p\right]
\\
&=
\mathbb{E} \left [ \sup_{t \in [0,T]}
\left|\int_0^t \int_{-d}^0e^{-\alpha t}
(\bar z^{x_1}_r(s)-\bar y^{x_1}_r(s))\, \phi(r)({\rm d}s)\, {\rm d}r \right|^p
\notag \right ]\\
&=
\mathbb{E}\left [ \sup_{t \in [0,T]}
\left|\int_0^t \int_{(-r)\vee(-d)}^0e^{-\alpha (t-r-s)}e^{-\alpha(r+s)}
|z(r+s)-y(r+s)|\, \tilde\phi(r)({\rm d}s)\, {\rm d}r\right|^p \right ]
\notag
\end{align}
where, in the last line, we used
the fact that $\tilde{\phi}(\omega)(ds) = \sup_{0\leq r\leq T}|{\phi}(r, \omega)| (ds)$.
Now we use the H\"older inequality and the estimate:
$$
\int_0^t \int_{[(-r)\vee(-d),0]}e^{-p'\alpha(t-r-s)} \phi(r)({\rm d}s)\, {\rm d}r
\le
\int_0^t \int_{[(-r)\vee(-d),0]}e^{-p'\alpha(t-r-s)} \tilde{\phi}({\rm d}s)\, {\rm d}r
$$
$$=
\int_{(-t)\vee (-d)}^0 \int_{-s}^t e^{-p'\alpha (t-r-s)} {\rm d}r
\tilde{\phi}({\rm d}s)
=\int_{(-t)\vee (-d)}^0 \frac{1}{p'\alpha}[e^{p'\alpha s}- e^{-p'\alpha t} ]
\tilde{\phi}({\rm d}s)
\le \frac{c^*_T}{p'\alpha}
$$
(here $c_T$ comes from Definition \ref{df:adm}) to get
\begin{align}
&\mathbb{E}
  \left[\sup_{t \in [0,T]} e^{-p\alpha t}
\left|\int_0^t (\mathcal N_T(\cdot)(\bar z^{x_1}-\bar y^{x_1})(r)\, {\rm d}r \right|^p\right]
\\
&\le
\mathbb{E}
\left[\sup_{t \in [0,T]}
\left(\frac{c_T}{p'\alpha}\right)^{p/p'}
\int_0^t \int_{(-r)\vee(-d)}^0e^{-p\alpha (r+s)}
|z(r+s)-y(r+s)|^p\, \tilde{\phi}(r)({\rm d}s)\, {\rm d}r\right]
\notag
\\
&\le
\mathbb{E}
\left[
\left(\frac{c_T}{p'\alpha}\right)^{p/p'}
\int_0^T \int_{(-r)\vee(-d)}^0
\sup_{(r+s) \in [0,T]}[e^{-p\alpha (r+s)}
|z(r+s)-y(r+s)|^p] \, \tilde\phi(r)({\rm d}s)\, {\rm d}r\right]
\notag\\
&\le
\left(\frac{c_T}{p'\alpha}\right)^{p/p'}
 T c_T \|z-y\|^p_\alpha
\notag
\end{align}
where we replaced $t$ with $T$ by monotonicity.

We now estimate the third term of \eqref{eq:alfanewestimate}
using the so-called factorization method.
Using, e.g., \cite[Lemma 1.114]{FABBRI_GOZZI_SWIECH_BOOK}) we can rewrite,
for $\eta \in (1/p,1/2)$ the stochastic integral of the third term of \eqref{eq:alfanewestimate} as follows 
$$
\int_0^{t} (z(r)-y(r))\sigma_y ^\top\, {\rm d}Z(r)=
c_\eta\int_{0}^{t}(t-u)^{\eta-1} Y(u) \, {\rm d}u
$$
where
$$
c_\eta^{-1}:=\int_{r}^{t}(t-u)^{\eta-1}(u-r)^{-\eta}du
=\frac{\pi}{\sin(\eta \pi)} \quad and \quad Y(u)=\int_{0}^{u}(u-r)^{-\eta}(z(r)-y(r))\sigma_y ^\top\, {\rm d}Z(r).
$$
Hence, applying the Holder inequality, we get, $\P$-a.s.
$$
e^{-\alpha t}\left|\int_0^{t} (z(r)-y(r))\sigma_y ^\top\, {\rm d}Z(r)\right|
=
c_\eta\left|\int_0^{t}
e^{-\alpha (t-u)}(t-u)^{\eta-1}e^{-\alpha u} Y(u) {\rm d}u\right|
$$
$$\le
c_\eta\left(\int_{0}^{t}(t-u)^{p'(\eta-1)}
e^{-p'\alpha (t-u)} du\right)^{1/p'}
\left(\int_{0}^{t} e^{-p\alpha u}|Y(u)|^p du\right)^{1/p}
$$
Hence
$$
\mathbb{E}\left[\sup_{t \in [0,T]} e^{-p\alpha t}\left|\int_0^{t} (z(r)-y(r))\sigma_y ^\top\, {\rm d}Z(r)\right|^p\right]
$$
$$
\le c_\eta^p\,
\mathbb{E}\left[\sup_{t \in [0,T]}
\left(\int_{0}^{t}(t-u)^{p'(\eta-1)}
e^{-p'\alpha (t-u)} du\right)^{p/p'}
\left(\int_{0}^{t} e^{-p\alpha u}|Y(u)|^p du\right)\right]
$$
$$
\le c_\eta^p\,
\mathbb{E}\left[
\left(\int_{0}^{T}u^{p'(\eta-1)}
e^{-p'\alpha u} du\right)^{p/p'}
\left(\int_{0}^{T} e^{-p\alpha u}|Y(u)|^p du\right)\right]
$$
Now, take out of the expectation the deterministic term. Apply Fubini's theorem to $d\P\otimes dt$,  and focus on
$e^{-\alpha u} \E|Y(u)|^p$.
By the Burkh\"{o}lder-Davis-Gundy inequality, we get for all $u \in [0,T]$,
\begin{align*}
e^{-\alpha p  u}\,  \E|Y(u)|^p  & \leq
k_p \,
e^{-\alpha p  u} \, \E\left(\int_{0}^{u}(u-r)^{-2\eta}(z(r)-y(r))^2\|\sigma_y\|^2 {\rm d}r\right)^{p/2}\\
& \le k_p\|\sigma_y\|^p\, \E\left(\int_{0}^{u}(u-r)^{-2\eta} e^{-2\alpha (u-r)}
\left[e^{-2\alpha r}(z(r)-y(r))^2\right] {\rm d}r\right)^{p/2} \\
& \le
k_p\|\sigma_y\|^p\, \E\left(\sup_{r\in [0,T]}\left(e^{-2\alpha r}(z(r)-y(r))^2\right)
\int_{0}^{u}(u-r)^{-2\eta} e^{-2\alpha (u-r)}
 {\rm d}r\right)^{p/2}
\\
& \le
k_p\|\sigma_y\|^p \left(\int_{0}^{u}(u-r)^{-2\eta} e^{-2\alpha (u-r)}
 {\rm d}r\right)^{p/2}
 \|z-y\|_\alpha^{ p}
\end{align*}
which implies
$$
\mathbb{E}\left[\sup_{t \in [0,T]}e^{-p\alpha t}\left|\int_0^{t} (z(r)-y(r))\sigma_y ^\top\, {\rm d}Z(r)\right|^p\right]
$$
$$
\le c_\eta^p\,
\left(\int_{0}^{T}u^{p'(\eta-1)}
e^{-p'\alpha u} du\right)^{p/p'}
T\, k_p\|\sigma_y\|^p
\left(\sup_{0\leq u\leq T}\int_{0}^{u}(u-r)^{-2\eta} e^{-2\alpha (u-r)}
 {\rm d}r\right)^{p/2}
\|z-y\|^p_\alpha
$$
Putting the three estimate above into \eqref{eq:alfanewestimate}
we get
\begin{align}
\|F(z)-F(y)\|_{\alpha}^p
\le C_T(\alpha) \|z-y\|_{\alpha}^p
\end{align}
where $C_T(\alpha) \to 0$ as $\alpha \to +\infty$. Thus, for $\alpha$ large enough,  $F$ is a contraction and therefore it admits  a unique
fixed point. This proves existence and uniqueness in the space $L^p(\Omega,S_T)$ for $p >2$. Since, for such $p$,
$L^p(\Omega,S_T)\subset L^2(\Omega,S_T)$,
such solution clearly belongs to $L^2(\Omega,S_T)$.

To get uniqueness in the space $L^2(\Omega,C([0,T];\R))$, proceed as follows.

\begin{itemize}
  \item  We just showed that  $F$ is an endomorphism of  $L^p(\Omega,C([0,T];\R))$, $p>2$ (and a contraction). However, $F$ is also an endomorphism of  $L^2(\Omega,C([0,T];\R))$, i.e. $y \in L^2(\Omega,C([0,T];\R))$ implies $F(y)\in L^2(\Omega,C([0,T];\R))$.  This can be proved along the same lines used before, by setting $p=2$, $z$ and $\alpha$ null. The first two terms of $F(y)$ go in the same way.   The $L^2$ estimate of the third term in \eqref{eq:alfanewestimate} is also straightforward, it does not rely on the factorization method but only on the Burkh\"{o}lder-Davis-Gundy inequality and It\={o} isometry. In fact, applying first the inequality and then the isometry, we get
      $$ \mathbb{E}\left[\sup_{t \in [0,T]} e^{-2\alpha t}\left|\int_0^{t} (z(r)-y(r))\sigma_y ^\top\, {\rm d}Z(r)\right|^2\right] \leq c_2\, E[ \int_0^T  (z(r)-y(r))^2 \|\sigma_y \|^2 dr]$$
      where $c_2$ is the appropriate constant from the BDG theorem.

  \item  Given this, suppose $y^a, y^{b}$ are  two solution in $L^2(\Omega,C([0,T];\R))$.  Consider the difference $d = y^a - y^{b} $.
    Then, $d$ satisfies
        $$ d(t) = \mu_y\int_0^t d(r)\, {\rm d}r+\int_0^t \mathcal{N}_T(\cdot) (\bar y^b)^{x_1}) - (\bar y^a)^{x_1})(r) {\rm d}r + \int_0^t d(r)\sigma_y ^\top\, {\rm d}Z(r), \quad \ 0 \le t \le T, $$
     Call $d^*$ the maximal functional of $d$:
      $$d^*(t) : = \sup_{s\in [0,t]} d(s) $$
      so that
      $$d(t) \leq \int_0^t d^*(r) (|\mu_y| + c_T ){\rm d}r + \int_0^t d(r)\sigma_y ^\top\, {\rm d}Z(r)  $$
      in which $ c_T$ is the constant from Definition \ref{df:adm}.
      Setting  $A(t) =  (\mu_y + c_T )t, M(t) =  \int_0^t d(r)\sigma_y ^\top\, {\rm d}Z(r), H=0$ in the statement of  stochastic Gronwall Lemma in \cite{Wiki} (see also \cite{Sima}), the above inequality leads to
$ \E[\sup_{t\in[0,T]}|d(t)|^2]=0$.
\end{itemize}
To show the continuous dependence of the solution of the SDDE on the datum $x$, by Theorem 7.1.1 in \cite{DAPRATO_ZABCZYK_RED_BOOK} it is sufficient to check the continuity of the integral function
$ F(y):=F(y;x)$ in \eqref{eq:integral-map} with respect to $x$. To this end,  fix a $y \in L^2(\Omega; C([0,T];\mathbb{R}))$ and consider any two data $ x,x^*\in M_2$.  Take then the difference
$$
\left(F(y;x^*)-F(y;x)\right)(t) =x_0^*-x_0 + \int_0^t \mathcal{N}_T(\cdot)( \bar{y}^{x^*}- \bar{y}^{x})(r))dr  $$
By Lemma \ref{newlemma.} the delay operator verifies an $L^2([-d,T];\mathbb{R})$ inequality, hence
\begin{align*}
\| F(y;x^*)-F(y;x) \|_{L^2(\Omega; C([0,T];\mathbb{R}))} & \leq |x_0^*-x_0| + c_T \| \bar{y}^{x^*}- \bar{y}^{x}\|_{L^2([-d,T];\mathbb{R})}\\
& = |x_0^*-x_0| + c_T  \| x_1^*-x_1\|_{L^2([-d,0];\mathbb{R})}
\end{align*}
where the last equality follows from $\bar{y}^{x^*} = \bar{y}^{x}=y$ on $[0,T]$.
\end{proof}


\begin{thebibliography} 9

\bibitem{Abowd-Card} Abowd, J. M., and D. Card.\emph{ On the Covariance Structure of Earnings and Hours
Changes.} Econometrica, 57(2), 411-445, 1989.

\bibitem{AKSAMITJEANBLANC17} Aksamit, A. and Jeanblanc M.
\emph{Enlargement of Filtration with Finance in View}.
Springer Briefs in Quantitative Finance, Springer, 2017

\bibitem{AB}
Aliprantis, C. D., and Border, K. C.
\emph{Infinite dimensional analysis: a Hitchhiker's Guide}. Third Edition, Springer, 2006.

\bibitem{BP} Biagini, S. and Pinar, M. \emph{The robust Merton problem of an ambiguity averse investor}.  Mathematics and Financial Economics,  (1), 2017.

\bibitem{BENSOUSSAN_DAPRATO_DELFOUR_MITTER} Bensoussan, A., Da Prato, G., Delfour, M.C., and Mitter, S.K. (2007)
\emph{Representation and Control of Infinite Dimensional Systems}, Second Edition, Birkhauser

\bibitem{BGPZ} Biffis, E., Goldys, B., C. Prosdocimi and M. Zanella (2019). \emph{A pricing formula for delayed claims: Appreciating the past to value the future}. Working paper Arxiv:
https://arxiv.org/abs/1505.04914

\bibitem{BGP} Biffis, E., Gozzi F., Prosdocimi C (2020).\emph{ Optimal portfolio choice with path dependent labor income: the infinite horizon case.} SIAM Journal on Control and Optimization, 58(4), 1906-1938.
 \bibitem{Brezis} H. Brezis 2011. \emph{Functional Analysis,
Sobolev Spaces and Partial Differential Equations}. Springer.

\bibitem{macche} S. Cerreia-Vioglio, F. Maccheroni, M. Marinacci and L. Montrucchio.  \emph{  Uncertainty averse preferences.}  Journal of Economic Theory
Volume 146, Issue 4, July 2011, Pages 1275-1330.

\bibitem{CHOJNOWSKA-MICHALIK_1978} Chojnowska-Michalik A. (1978),
\emph{Representation Theorem for General Stochastic Delay Equations.}
in {\slshape Bull. Acad. Polon. Sci.S\'{e}r. Sci. Math. Astronom. Phys.},
\textbf{26} 7, pp. 635-642.

\bibitem{CFGRT} A. Cosso, S. Federico, F. Gozzi, M. Rosestolato and N. Touzi.  \emph{  Path-dependent equations and viscosity solutions in infinite dimension.}  Annals of Probability
Volume 46, Issue 1 (2018), Pages 126-174.


\bibitem{DAPRATO_ZABCZYK_RED_BOOK} Da Prato, G. and Zabczyk, J. (2014),
\emph{Stochastic equations in Infinite Dimensions}. Cambridge University Press, Second Edition.


 \bibitem{DUNS-Go-Tran} Dunsmuir, W.T., Goldys, B., and C.V. Tran (2016).\emph{ Stochastic delay differential equations
as weak limits of autoregressive moving average time series.}  Working paper, University of New
South Wales.

\bibitem{DYBVIG_LIU_JET_2010} Dybvig, P.H. and Liu, H. (2010). \emph{ Lifetime consumption and investment: retirement and constrained borrowing}. {\slshape Journal of Economic Theory}, 145, pp. 885-907.

\bibitem{DYBVIG_LIU_JET_2010} Dybvig, P.H. and Liu, H. (2010). \emph{ Lifetime consumption and investment: retirement and constrained borrowing}. {\slshape Journal of Economic Theory}, 145, pp. 885-907.
\bibitem{DGZZ} Djeiche, B. Gozzi, F. Zanco, G. and Zanella, M. (2022). \emph{Optimal portfolio choice with path dependent benchmarked labor income: a mean field model}. {\slshape Stochastic Processes and Applications}, 145 (2022), pp.48-85.

\bibitem{Hahn} Encyclopedia of Mathematics, EMS Press, 2001. Item  "Hahn decomposition".

\bibitem{FABBRI_GOZZI_SWIECH_BOOK} Fabbri, G. Gozzi, F. and Swiech, A. (2017).
\emph{Stochastic Optimal Control in Infinite Dimensions: Dynamic Programming and HJB Equations}. Probability Theory and Stochastich Modelling, vol 82, Springer.
\bibitem{FS} Foellmer, H. and Schweizer, M. (2010).
\emph{Minimal Martingale Measure}.    Encyclopedia of Quantitative Finance, Wiley, 1200-1204.

\bibitem{Flandoli90SICON} Flandoli, F. (1990).
\emph{Lifetime consumption and investment: retirement and constrained borrowing.}
{  Journal of Economic Theory}, 145, pp. 885-907.

\bibitem{FlaZan}
F.~Flandoli and G.~Zanco.
\newblock \emph{An infinite-dimensional approach to path-dependent {K}olmogorov
  equations.}
\newblock { Ann. Probab.}, 44(4), 2643--2693, 2016.

\bibitem{JYC}    Jeanblanc, M., Yor,
M., Chesney,  (2009).
\emph{Mathematical Methods for Financial Markets},
Springer-Verlag.


\bibitem{HaddSICON06}
S.~Hadd, \newblock  \emph{An evolution equation approach to nonautonomous linear systems with state, input, and output delays. }
\newblock{ SIAM, Journal on control and optimization}, 45(1):246--272, 2006.

\bibitem{Kallenberg} Kallenberg O. (1997)
\newblock  \emph{Foundations of
Modern Probability
}, Springer-Verlag.

\bibitem{KARATZAS_SHREVE_91} Karatzsas, I. and Shreve, S.E. (1991).
\emph{Brownian Motion and Stochastic Calculus}, Springer-Verlag

\bibitem{riedel}Q. Lin and F. Riedel, 2014, {\em Optimal Consumption and Portfolio Choice
with Ambiguity}, Working paper, Center for Mathematical Economics, University of Bielefeld.

\bibitem{LORENZ_2006} Lorenz, R. (2006)
\emph{Weak Approximation of Stochastic Delay
	Differential Equations with Bounded Memory by
	Discrete Time Series}. PhD dissertation, Humboldt University.

 \bibitem{McLaughling} K.J. McLaughling, \emph{Wage rigidity?}  (1993), Journal of Monetary Economics,
34(3), pp. 383–414.

\bibitem{Mao-Sabanis} X. Mao and S. Sabanis, \emph{ Delay geometric Brownian motion in financial
option valuation} (2013)- Stochastics: An International Journal of Probability and
Stochastic Processes, 85(2), pp. 295–320.

\bibitem{MEGHIR_PISTAFERRI_2004} Meghir, C., Pistaferri, L. (2004). \emph{Income variance dynamics and heterogeneity.} { Econometrica}, 72(1), 1-32.




\bibitem{MERTON} Merton, R. (1990). \emph{Continuous-time finance}. Basil Blackwell, Oxford.

\bibitem{MOHAMMED_BOOK_96} Mohammed, SE.A.
\emph{Stochastic Differential Systems with Memory: Theory, Examples and Applications}. In: Decreusefond L.,{\O}ksendal B., Gjerde J., \"{U}st\"{u}nel A.S. (eds) Stochastic Analysis and Related Topics VI. Progress in Probability, vol 42. Birkh\"auser, Boston, MA.

\bibitem{NN}A. Neufeld, M. Nutz (2018).
\emph{Robust Utility Maximization with Lévy Processes}.
Mathematical Finance, Vol. 28, No. 1, pp. 82-105.

\bibitem{Protter} Protter, P.E. (2005)
\emph{Stochastic Integration and Differential Equations}, Springer-Verlag Berlin.

\bibitem{RenRosest} Z. Ren and M. Rosestolato.  \emph{  Viscosity solutions of path-dependent PDEs with randomized time.}
SIAM Journal on Mathematical Analysis 52 (2), 1943-1979.

\bibitem{REISS_2002} Rei{\ss}, M. (2002).
\emph{Nonparametric estimation for stochastic delay differential equations}. PhD dissertation, Humboldt University.

\bibitem{Rosestolato17}
Rosestolato, M., (2017).
\emph{Path-dependent SDEs in Hilbert spaces.}
{ International Symposium on BSDEs, 261-300.}


\bibitem{RosestolatoSwiech17}
Rosestolato, M., Swiech, A., (2017).
\emph{Partial Regularity of Viscosity Solutions for a class of Kolmogorov Equations arising from Mathematical Finance.}
{ J. Differential Equations 262 (2017), no. 3, 1897--1930.}

\bibitem{Sima} Mehri, Sima; Scheutzow, Michael (2021).
\emph{A stochastic Gronwall lemma and well-posedness of path-dependent SDEs driven by martingale noise}. Latin American Journal of Probability and Mathematical Statistics. 18: 193-209.

\bibitem{TranPHD} C.V. Tran (2016).\emph{ Convergence of Time Series Processes to Continuous Time Limits.}
     PHD dissertation, University of New South Wales.
     Available at
     "http://unsworks.unsw.edu.au/fapi/datastream/unsworks:11500/SOURCE01?view=true"

  \bibitem{VICERIA} Viceira, L. M. (2001) \emph{Optimal Portfolio Choice for Long-Horizon
Investors with Nontradable Labor Income}. {\slshape The Journal of Finance, LVI, no. 2 pp. 433-470.}


\bibitem{VINTER} Vinter, R. B.  (1975)
\emph{A representation of solution to stochastic delay equations}, Imperial College, Report of the Department of Computing and Control.
\bibitem{Wiki} Wikipedia, item \href{ https://en.wikipedia.org/wiki/Stochastic_Gronwall_inequality}{Stochastic Gronwall Lemma}.
\end{thebibliography}
\end{document}